\documentclass[12pt]{amsart}
\usepackage{amsmath, amsthm, amsfonts, amssymb}
\usepackage{mathtools} 
\usepackage[utf8]{inputenc}
\usepackage[all,cmtip]{xy}
\usepackage{url}

\usepackage{eucal}
\usepackage{CJKutf8}
\usepackage{marvosym}

\usepackage{tikz}
\usepackage{tikz-cd}
\usetikzlibrary{decorations.markings}

\theoremstyle{plain}
\newtheorem{thm}{Theorem}[section]
\newtheorem{lemma}[thm]{Lemma}
\newtheorem{prop}[thm]{Proposition}
\newtheorem{cor}[thm]{Corollary}

\theoremstyle{definition}
\newtheorem{defn}[thm]{Definition}

\newtheorem{rem}[thm]{Remark}

\numberwithin{equation}{section}

\newcommand{\squarething}[1]{
  \noindent
  \begin{center}
  \framebox{
    \vbox{
      \vspace{4mm}
      \hbox to 5.78in { {\Large \hfill #1  \hfill} }
      \vspace{4mm}
    }
  }
  \end{center}
  \vspace*{4mm}
}

\advance \topmargin by -\headheight
\advance \topmargin by -\headsep
\evensidemargin \oddsidemargin
\DeclareMathOperator{\Var}{Var}

\DeclareMathOperator{\FPdim}{FPdim}

\DeclareMathOperator{\id}{id}

\DeclareMathOperator{\ad}{ad}

\DeclareMathOperator{\Gal}{Gal}

\DeclareMathOperator{\tr}{Tr}

\DeclareMathOperator{\gr}{gr}

\DeclareMathOperator{\Vir}{Vir}

\newcommand{\mfl}{\mathfrak{l}}
\newcommand{\laset}{\widehat{X}}
\newcommand{\beset}{X}

\newcommand{\al}{\alpha}
\newcommand{\la}{\lambda}
\newcommand{\La}{\Lambda}

\newcommand{\D}{\Delta}

\newcommand{\C}{\mathbb{C}}

\newcommand{\Q}{\mathbb{Q}}
\newcommand{\R}{\mathbb{R}}
\newcommand{\Z}{\mathbb{Z}}

\newcommand{\g}{\mathfrak{g}}
\newcommand{\h}{\mathfrak{h}}
\newcommand{\n}{\mathfrak{n}}

\newcommand{\sll}{\mathfrak{sl}}

\newcommand{\CF}{\mathcal{F}}

\newcommand{\alstar}{{\alpha_*}}
\newcommand{\fing}{\g}

\newcommand{\vac}{\mathbf{1}}

\newcommand{\Pvf}[1]{{\check{P}_{+, f}^{#1}}}

\newcommand{\wtil}[1]{\widetilde#1}
\newcommand{\what}[1]{\widehat#1}
\newcommand{\ov}[1]{\overline#1}

\newcommand{\mbo}{\mathbb{O}}

\DeclareMathOperator{\prin}{Pr}

\makeatletter
\def\@settitle{\begin{center}%
  \baselineskip14\p@\relax
    \normalfont\LARGE
  \@title
  \end{center}%
}
\makeatother

\title[Characters and fusion rules]{Characters and fusion rules of boundary $W$-algebras}

\subjclass[2010]{}

\author{Tomoyuki Arakawa\textsuperscript{1}}
\address{\textsuperscript{2}Research Institute for Mathematical Sciences\\ Kyoto University\\ Kyoto 606-8502 JAPAN}
\email{arakawa@kurims.kyoto-u.ac.jp}

\author{Igor Alarcon Blatt\textsuperscript{2}}
\address{\textsuperscript{2}Instituto de Matem\'{a}tica Pura e Aplicada\\ Jardim Bot\^{a}nico, RJ 22.460-320 BRAZIL}
\email{igor.blatt@impa.br}

\author{Jethro van Ekeren\textsuperscript{2}}
\address{\textsuperscript{2}Instituto de Matem\'{a}tica Pura e Aplicada\\ Jardim Bot\^{a}nico, RJ 22.460-320 BRAZIL}
\email{jethro@impa.br}

\author{Wenbin Yan\textsuperscript{3}}
\address{\textsuperscript{3}Yau Mathematical Sciences Center\\ Tsinghua University, Beijing, CHINA}
\email{wbyan@tsinghua.edu.cn}

\begin{document}

\maketitle

\vspace*{5mm}

\noindent
\textbf{Abstract.} We study the $q$-characters and modular data of exceptional $W$-algebras and give several examples and applications. We establish equality of $q$-characters and modular data between certain boundary $W$-algebras, leading in particular to a largely complete determination of fusion rules of exceptional $W$-algebras in type $A$.


\section{Introduction}

The affine $W$-algebras constitute a large and important class of vertex algebras, with applications in integrable systems, representation theory, enumerative geometry, and theoretical physics \cite{FF, KRW03, BD-Hitchin, SV2013, AGT}. Categories of representations of exceptional $W$-algebras yield examples of modular tensor categories, and it is an interesting problem to determine the structure of these categories \cite{FKW, AE2023}.

Among all exceptional $W$-algebras, the so called boundary $W$-algebras are particularly important for several reasons. It is these $W$-algebras which arise in the 4d/2d correspondence \cite{BLLPRV}, and the $q$-characters of these algebras and their irreducible modules are given by interesting infinite product formulas \cite{XYY, KW.boundary}. In this article the importance of the boundary $W$-algebras is due to their role in controlling the fusion rules of exceptional $W$-algebras more broadly \cite{AE2023}, as we now briefly explain.

Let $\g$ be a simple Lie algebra, $\mbo \subset \g$ a nilpotent orbit and $k = -h^\vee + p/u$ an admissible level for $\g$ (see Section \ref{sec:background} below for definitions and background material). We denote by $W(\mbo, p/u)$ the simple quotient of the quantum Drinfeld-Sokolov reduction of the universal affine vertex algebra $V^k(\g)$ relative to a nilpotent element $f \in \mbo$. There is a certain nilpotent orbit $\mbo = \mbo_u$, depending on the denominator $u$, for which $W(\mbo, p/u)$ is lisse \cite{Arakawa2015} and rational \cite{AE2023, McRae2021}; we refer to such vertex algebras as \emph{exceptional $W$-algebras}. The \emph{boundary $W$-algebras} are those for which $p$ assumes its smallest possible value $p = h^\vee$ (we assume for simplicity here that $\g$ is simply laced and $\gcd(h^\vee, u) = 1$). The general pattern is that the fusion rules of $W(\mbo_u, p/u)$ are obtained as products of fusion rules of the corresponding boundary $W$-algebra $W(\mbo_u, h^\vee/u)$ and of the simple affine vertex algebra $L_{p-h^\vee}(\g)$. The fusion rules of the simple affine vertex algebras are well known, so the determination of fusion rules of exceptional $W$-algebras in general, largely reduces to the boundary case. See Corollary \ref{cor:intro} and Section \ref{sec:type.A.moddata} below for precise statements and details.

In Proposition \ref{prop:S-matrix-boundary} we give a formula for the $S$-matrix of a boundary $W$-algebra, under some additional hypotheses explained in Section \ref{subsec:aff.W} below. These hypotheses are satisfied for all nilpotent orbits $\mbo_u$ in type $A$ (and outside type $A$ for some nilpotent orbits but not others). We apply the formula to determine the modular data, hence fusion rules, of all boundary $W$-algebras in type $A$. Indeed we prove the following theorem.
\begin{thm}\label{thm:main.thm.typeA}
Let $n = um+s$ where $1 \leq s \leq u-1$ and $s$ is coprime to $u$. Concerning the exceptional $W$-algebra $W(\sll_n[u^m, s], n/s)$ and the principal $W$-algebra $W(\sll_s[s], s/u)$, the following statements hold:
\begin{enumerate}
\item The irreducible modules of these vertex algebras are in bijection,

\item under this bijection corresponding modules have the same $q$-character,

\item under this bijection the vertex algebras have the same modular data.
\end{enumerate}
\end{thm}
This result strongly suggests that the vertex algebras mentioned are in fact isomorphic, as was conjectured in {\cite[Conjecture 8.11]{AEM2024}} and {\cite{XY2021}} (see also \cite{CFLN2025, GJ2025, BN2025}). It should be noted, however, that there exist pairs of vertex algebras for which all three of the statements in Theorem \ref{thm:main.thm.typeA} hold, and which nevertheless are non isomorphic. For example the two lattice vertex algebras $V_{N}$ and $V_{N'}$ where $N$ is the Niemeier lattice containing the root lattice $E_6^4$, and $N'$ the Niemeier lattice containing the root lattice $A_{11} D_7 E_6$.

The utility of Theorem \ref{thm:main.thm.typeA} extends beyond the boundary case, due to the factorisation phenomenon alluded to above. To state our main corollary on fusion rules in type $A$, we introduce some notation. For a vertex algebra $V$ we denote by $\CF(V)$ the Grothendieck group of its category of representations. If $V$ is assumed rational, lisse, self-dual (and other minor technical hypotheses) then the category of representations carries a modular tensor category structure \cite{Huang-rigidity}, and so $\CF(V)$ becomes a commutative ring.
\begin{cor}\label{cor:intro}
Let $n = um+s$ where $1 \leq s \leq u-1$ and $s$ is coprime to $u$. Assume further that $u$ and $s$ are odd. Then
\[
\CF(W(\sll_n[u^m, s], p/u)) \cong \CF(L_{u-s}(\sll_s))^{\text{int}} \otimes \CF(L_{p-n}(\sll_n)).
\]
\end{cor}

As noted above, the main technical results (Propositions \ref{prop:H.la.char.thm} and \ref{prop:S-matrix-boundary}) hold in some circumstances outside type $A$ also. In Section \ref{subsec:vir.prods} we apply the results in type $D$ we obtain a conceptual explanation of an observation of \cite{BF2000} and \cite{FFW} on product formulas for certain sums of minimal model characters. We also present the modular data of an interesting modular tensor category which arises as a type $E_8$ subregular exceptional $W$-algebra.

\emph{Acknowledgements}: The authors would like to thank T. Gannon, A. Moreau and S. Nakatsuka for useful discussions. JE was supported by grants CNPq 306498/2023-5 and Serrapilheira 2023-0001. The results in this article were announced at the Simons center program ``Supersymmetric Quantum Field Theories, Vertex Operator Algebras, and Geometry'' in April 2025, we would like to thank the organisers of this stimulating program.

\section{Background and notation}\label{sec:background}

Let $\g$ be a finite dimensional simple Lie algebra and $f \in \g$ a nilpotent element. The affine $W$-algebra $W^k(\g, f)$ is defined through quantum Drinfeld-Sokolov reduction of the affine vertex algebra $V^k(\g)$, explicitly as the cohomology of a certain differential graded vertex algebra modeled on the BRST prescription \cite{FF, KRW03}. The simple quotient $W_k(\g, f)$ is of interest when the level $k = -h^\vee + p/u$ is admissible \cite{KW88}, particularly when $f$ lies in a certain nilpotent orbit $\mbo_u$ determined by $k$ (or rather by its denominator $u$) \cite{Arakawa2015}, in which case $W_k(\g, f)$ is rational \cite{McRae2021}. We refer to these vertex algebras as exceptional $W$-algebras \cite{KW2009, AE2023}. In this section we review these key notions.

\subsection{Affine Kac-Moody algebras and affine vertex algebras}

We fix a triangular decomposition $\g = \n_- + \h + \n_+$ of the finite dimensional simple Lie algebra $\g$ and denote by $\D$ the set of roots, $\D_+$ the set of positive roots and $\Pi$ the set of simple roots. The non degenerate invariant bilinear form $(\cdot, \cdot) : \h \times \h \rightarrow \C$ induces a form on $\h^*$, which we assume normalised so that the highest root $\theta$ has norm $2$. The spaces $\h$ and $\h^*$ are identified via this form, and the coroot $\al^\vee$ corresponding to $\al \in \D$ is then $\al^\vee = 2\al / (\al, \al)$. The set of coroots is denoted $\D^\vee$.

We have the root lattice $Q = \sum_{\al \in \D} \Z \al \subset \h^*$ and coroot lattice $\check{Q} = \sum_{\al \in \D^\vee} \Z \al^\vee \subset \h$, as well as their canonical duals the weight lattice $P = \check{Q}^* \subset \h^*$ and coweight lattice $\check{P} = Q^* \subset \h$. If $\g$ has rank $\ell$ and $\Pi = \{\al_1, \ldots, \al_\ell\}$ then the set of fundamental weights $\{\varpi_1, \ldots, \varpi_\ell\}$ is by definition the canonical dual basis of $P$, and the fundamental coweights $\varpi_i^\vee$ are defined similarly. The dominant integral weights are elements of the cone $P_+ = \sum \Z_+ \varpi_i$, and the set of dominant integral coweights $\check{P}_+$ is defined similarly.

The Weyl group is denoted $W$, the Weyl vector $\rho$ (equal to the sum of the fundamental weights, and also to half the sum of the positive roots). The number $h^\vee = 1 + (\rho, \theta)$ is the dual Coxeter number of $\g$. We denote by $r^\vee$ the lacing number of $\g$.

The untwisted affine Kac-Moody algebra $\what\g$ associated with $\g$ is
\[
\what\g = \g[t, t^{-1}] \oplus \C K \oplus \C d
\]
where
\[
[at^m, bt^n] = [a, b] t^{m+n} + m \delta_{m, -n} (a, b) K, \quad [K, \what\g] = 0, \quad [d, at^m] = m at^m.
\]
A Cartan subalgebra is given by $\what\h = \h + \C K + \C d$ and its dual $\what\h^* = \h^* + \C \La_0 + \C \delta$, where $\La_0(K) = \delta(d) = 1$ and $\La_0(D) = \delta(K) = 0$, is identified with $\what\h$ through the invariant bilinear form giving $(\La_0, \La_0) = (\delta, \delta) = 0$ and $(\La_0, \delta) = 1$.

The affine root system is the union of $\what\D^{\text{re}} = \{\al + n \delta \mid \al \in \D, n \in \Z\}$ (the real roots) and $(\Z \backslash 0)\delta$ (the imaginary roots). The set $\what\Pi$ of simple roots consists of the finite roots $\al_i \in \Pi$ for $i=1,\ldots, \ell$, together with $\al_0 = -\theta + \delta$. The coroot associated with $\al \in \what\D^{\text{re}}$ is $\al^\vee = 2\al / (\al, \al)$, the set $\{\La_0, \ldots, \La_\ell\}$ of fundamental weights is defined as the dual basis to the set of simple coroots, and the set of dominant integral weights is $\sum \Z_+ \La_i$ as in the finite case.

The level of a weight $\what\la \in \what\h^*$ is the number $\what\la(K)$. It is easy to see that the set $\what{P}_+^k$ of dominant integral weights of level $k$ is
\[
P_+^k = \{ k\La_0 + \la \mid \la \in P_+^k\} \quad \text{where} \quad P_+^k = \{\la \in P_+ \mid (\la, \theta) \leq k \}.
\]
We denote by $P_+^{k, \text{reg}}$ the subset of regular weights in $P_+^k$, i.e., of weights not orthogonal to any simple root of $\what\g$. The affine Weyl vector is defined as $\what\rho = h^\vee \La_0 + \rho$.

The affine Weyl group $\what{W}$, generated by reflections in real roots, is isomorphic to $W \ltimes t_{\check{Q}}$ where $t_{\check{Q}}$ denotes the set of translations $t_\al$ for $\al \in \check{Q}$, defined by
\[
t_{\al}(\what\la) = \what\la + k \alpha - \left( (\what\la, \al) + k |\al|^2 / 2 \right) \delta \quad \text{where $k = \what\la(K)$}.
\]
The extended affine Weyl group $\wtil{W}$ is similarly defined as $W \ltimes t_{\check{P}}$. The root system $\what\D$ is stable under this larger group, and the quotient $\wtil{W}_+ = \wtil{W} / \what{W}$ is finite.

We denote by $M(\la)$ the Verma $\g$-module of highest weight $\la \in \h^*$ and $L(\la)$ its simple quotient, and $\what\g$-modules $M(\what\la)$ and $L(\what\la)$ similarly for $\what\la \in \h^*$. We also sometimes denote $L(k\La_0+\la)$ as $L_k(\la)$. For arbitrary level $k$ the irreducible highest weight $\what\g$-module $L(k\La_0)$ carries the structure of a vertex algebra, which we denote $L_k(\g)$ \cite{FZ92}. The vertex algebra $L_k(\g)$ is simple, and is realised as a quotient of the universal affine vertex algebra $V^k(\g)$, whose underlying $\what\g$-module is the vacuum module $U(\what\g) \otimes_{U(\g[t] + \C K)} \C \vac$ (in which the $\g[t]$-action on $\C\vac$ is trivial).

\subsection{Principal admissible weights}

The notions of admissible and principal admisisble weight were introduced in {\cite{KW88}}. For any weight $\what\la = k\La_0 + \la \in \what\h^*$ one has the associated integral coroot system
\[
\what\D^{\vee, \text{re}}(\what\la) = \{\al^\vee \in \what\D^{\vee, \text{re}} \mid \left< \what\la, \al^\vee \right> \in \Z\}.
\]
One then also has, as usual, the subset $\what\D^{\vee, \text{re}}(\what\la)_+ = \what\D^{\vee, \text{re}}(\what\la) \cap \what\D^{\vee, \text{re}}_+$, and the set $\what\Pi^\vee(\what\la)$ of simple coroots relative to $\what\la$, i.e., the set of elements of $\what\D^{\vee, \text{re}}(\what\la)_+$ which are not expressible as a sum of two or more elements of the same set.
\begin{defn}
A weight $\what\la = k\La_0 + \la$ is said to be \emph{admissible} if
\[
\left<\what\la + \what\rho, \al^\vee\right> \notin \Z_{\leq0}
\]
for all $\al^\vee \in \what\D^{\vee, \text{re}}$, equivalently if $\la$ is dominant integral with respect to $\what\D^{\vee, \text{re}}(\la)_+$, and if furthermore the span over $\Q$ of $\what\D^{\vee, \text{re}}(\what\la)$ coincides with that of $\what\D^{\vee, \text{re}}$. The admissible weight $\what\la$ is said to be \emph{principal admissible} if $\what\D^{\vee, \text{re}}(\what\la)_+$ is isometric to $\what\D^{\vee, \text{re}}_+$. We denote by $\prin^k$ the set of $\la \in \h^*$ such that $k\La_0 + \la$ is principal admissible.
\end{defn}
It is known {\cite{KW89}} that $\prin^k$ is non empty only for levels $k$ of the form
\[
k = -h^\vee + p/u, \quad \gcd(p, u)=\gcd(u,r^\vee)=1, \quad p \geq h^\vee
\]
known as principal admissible levels, that the weight $k\La_0$ is principal admissible precisely when $k$ is a principal admissible level, and that for $\la \in \prin^k$, there exists $\wtil{y} = yt_{-\eta} \in \wtil{W}$ such that $\what\D^{\vee}(\la)_+ = \wtil{y}(\what\Pi_u)$ where
\[
\what\Pi_u = \{uK - \theta^{\vee}_{\text{s}}, \al_1^\vee, \ldots, \al_n^\vee\}.
\]
(here $\theta^{\vee}_{\text{s}}$ is the highest short coroot of $\g$). Indeed this leads to the classification of principal admissible weights, which we paraphrase from \cite[Proposition 1.5]{KW90}
\begin{prop}
Each weight $\la \in \prin^k$ can be written in the form
\begin{align}\label{eq:prin.param}
\la = y(\nu - ({p}/{u})\eta) - \rho 
\end{align}
where $y \in W$, $\eta \in \check{P}_+^{u}$ and $\nu \in P_+^{p, \text{reg}}$. The finite group $\wtil{W}_+$ acts on the set of triples $(y, \eta, \nu)$, preserving $\la$ in \eqref{eq:prin.param}, and $\prin^k$ is in bijection with the set of orbits of this action.
\end{prop}

If $k$ is a principal admissible level then the $\what\g$-modules $L_k(\la)$ acquire the natural structure of modules over $L_k(\g)$ for $\la \in \prin^k$ (and for no other $\la \in \h^*$) \cite{A.rat.O}.

\begin{rem}
When $\g$ is non simply laced there arises the possibility that $\what\D^{\vee, \text{re}}(\what\la)_+$ be isometric not to $\what\D^{\vee, \text{re}}_+$ but rather to a certain twisted affine root system. Such weights are known as \emph{coprincipal admissible} \cite{A.rat.O}. Most statements quoted in Section \ref{subsec:aff.W} below were proved in the cited articles in both the principal and coprincipal cases, and most of the results of this article also hold in the coprincipal case. That said, our chief applications are in simply laced types.
\end{rem}

\subsection{Affine $W$-algebras}\label{subsec:aff.W}

Let $f \in \g$ be an $\ad$-nilpotent element. The universal affine $W$-algebra $W^k(\g, f)$ is defined via quantum Drinfeld-Sokolov reduction of $V^k(\g)$ \cite{KRW03}. We recall a few details about the construction. By the Jacobson-Morozov theorem there exists an $\sll_2$-triple $\{e, h, f\}$ containing $f$. Then we have $\g = \bigoplus_{j \in (1/2)\Z} \g_j$ the grading by eigenvalues of $\ad(h/2)$. In fact, as in \cite{KRW03}, the construction can be carried out in the more general setting of a \emph{good grading} of $\g$ compatible with $f$ \cite{EK}. In this setting the grading continues to correspond to the eigenspace decomposition relative to $\ad(x^0)$ for a grading element $x^0 \in \h$, only now $x^0$ might differ from $h/2$. The case $x^0 = h/2$ is then referred to as the Dynkin grading. At certain times we will assume our grading to be even (i.e., over $\Z$ rather than $(1/2)\Z$), so it is important to work with non Dynkin gradings. Note that we can choose our triangular decomposition of $\g$ compatible with the grading. We write $\D_0$ for the root system of $\g_0$ as well as $\D_{>0}$, etc., similarly.

We recall the vertex superalgebra $F(\g_{>0})$ of charged free fermions \cite[Section 2.5]{va.beginners} \cite[Example 1.3]{KRW03}, whose underlying vector superspace is the Grassmann algebra in variables $\varphi^{(i)}_n$ for $n \leq 0$ and $\varphi^{(i)*}_n$ for $n \leq -1$, where $\varphi^{(i)}$ runs over a basis of $\g_{>0}$ and $\varphi^{(i)*}$ denotes the dual basis of $\g_{>0}^*$. There is a $\Z$-grading induced by $\deg(\varphi^{*}) = +1$ and $\deg(\varphi) = -1$ and a certain differential (the BRST charge) $d$, and by definition
\[
W^k(\g, f) = H^0(V^k(\g) \otimes F(\g_{>0}), d).
\]
The standard notation for the simple quotient of $W^k(\g, f)$ is $W_k(\g, f)$. However we will almost always be working with levels of the form $k = -h^\vee + p/u$, and so to avoid clutter as well as emphasise the role of the nilpotent orbit $\mbo = G \cdot f$ we adopt the notation $W(\mbo, p/u)$ for the simple quotient.

We recall the nilpotent orbit $\mbo_u \subset \g$, characterised by  its Zariski closure
\begin{align}\label{eq:mbo.u}
\ov\mbo_u = \{x \in \g \mid \ad(x)^{2u} = 0\},
\end{align}
in case $\gcd(u, r^\vee) = 1$ and by a similar condition, which we do not recall, in the coprincipal case $\gcd(u, r^\vee) = r^\vee$.
\begin{defn}
The simple quotient of the vertex algebra $W^k(\g, f)$ is called an \emph{exceptional $W$-algebra} if the level $k = -h^\vee + p/u$ is admissible and $f \in \mbo_u$.
\end{defn}
The exceptional $W$-algebras are known to be lisse {\cite{Arakawa2015}} and rational {\cite{McRae2021}}. Under some additional hypotheses (see Theorem \ref{thm:rationality.thm} below), the classification of irreducible modules is also understood. We note that for such cases the rationality of the vertex algebra was established already in \cite{AE2023}. Irreducible $W(\mbo, p/u)$-modules may be constructed by applying quantum Drinfeld-Sokolov reduction to appropriate $\what\g$-modules. For technical reasons it is more convenient to work with a slightly different reduction called the $(-)$-variant, denoted $H_{f, -}^0(-)$ \cite{FKW} \cite{Arakawa.Rep.I}. 

We now summarise some results about the construction of modules over exceptional $W$-algebras via the $(-)$-variant reduction. The notion of associated variety plays an important role. The associated variety $\Var(I)$ of an ideal $I \subset U(\g)$ is by definition the maximal spectrum of the image of the associated graded $\gr(I) \subset S(\g)$, viewed as a subvariety of the affine space $\g$. The Zhu algebra of $V^k(\g)$ is isomorphic to the universal enveloping algebra $U(\g)$ \cite{Zhu96, FZ92} and the Zhu algebra of $L_k(\g)$ is a quotient $U(\g) / I_k$ for some two-sided ideal $I_k$. If a $\what\g$-module $L_k(\la)$ is to descend to a module over the vertex algebra $L_k(\g)$, then $I_k$ must be contained in $J_\la \subset U(\g)$ the annihilator of $L(\la)$, and so $\Var(J_\la)$ must be contained in the closure of $\Var(I_k)$. In fact $\Var(I_k) = \overline{\mbo}_u$ by definition, the more explicit characterisation \eqref{eq:mbo.u} being {\cite[Theorem 5.14]{Arakawa2015}}.

Let $f \in \mbo_u$ now. It turns out {\cite[Theorem 5.5.4]{Arakawa.Rep.II}} that the reduction $H_{f, -}(-)$ annihilates those modules $L_k(\la)$ for which the containment $\Var(J_\la) \subset \ov{\mbo}_u$ is strict. We are thus led to formulate:
\begin{defn}
Call the weight $\la \in \prin^k$ \emph{replete} if $\Var(J_\la) = \overline{\mbo}_u$. Denote by $\prin^k_\circ$ the set of replete admissible weights of level $k$.
\end{defn}
Let $\sim$ denote the relation $\la' = y \circ \la$ for some $y \in W$, where $y \circ \la$ denotes the ``dot'' action $y(\la + \rho) - \rho$. For weights $\la, \la' \in \prin^k$ satisfying $\la' \sim \la$ the reductions of $L_k(\la')$ and $L_k(\la)$ are isomorphic. We denote by $P_{0, +}$ the set $\{\la \in \h^* \mid \text{$\left<\la, \al^\vee\right> \in \Z_+$ for all $\al \in \D_{0, +}$} \}$ of weights of $\g$ integrable with respect to $\D_{0, +}$. Summarising the discussion so far, we have:
\begin{prop}{\ }
Let $k = -h^\vee + p/u$ be a principal admissible level for $\g$.
\begin{enumerate}
\item  If $\la \in \prin^k \cap P_{0, +}$ then $H_{f, -}^0(L_k(\la))$ is an ordinary $W(\mbo, p/u)$-module.

\item If $\la, \la' \in \prin^k$ and $\la' \sim \la$, then $H_{f, -}^0(L_k(\la')) \cong H_{f, -}^0(L_k(\la))$.

\item For $\la \in \prin^k$ we have $H_{f, -}^0(L_k(\la)) = 0$ unless $\la \in \prin^k_\circ$.
\end{enumerate}
\end{prop}
We consider the hypothesis that each equivalence class of $\prin^k_\circ / \sim$ have a representative $\la \in P_{0, +}$. The reduction $H^0_{f, -}(L_k(\la))$ is an irreducible ordinary module over the rational $W$-algebra $W(\mbo_u, p/u)$, and our hypothesis is chosen to guarantee that all such modules are obtained this way.
\begin{thm}[{\cite[Theorem 7.9]{AE2023}}]\label{thm:rationality.thm}
Suppose every element of $\prin^k_\circ / \sim$ possesses a representative $\la \in P_{0, +}$, and let $\laset$ be a set consisting of one such representative $\la$ of each equivalence class. Then $\{H_{f, -}^0(L_k(\la)) \mid \la \in \laset\}$ is a complete set of isomorphism classes of irreducible modules over the lisse rational vertex algebra $W(\mbo_u, p/u)$.
\end{thm}
We remark that the hypothesis in Theorem \ref{thm:rationality.thm} is satisfied in many cases, including whenever $\mbo_u$ is (1) the principal nilpotent orbit for $\g$ of any type, (2) any nilpotent orbit in $\g$ of type $A$, and (3) the subregular nilpotent orbit in $\g$ of simply laced type.

We now record an easy lemma which will be useful later on.
\begin{lemma}\label{lem:replete.integrability.general}
Let $\la \in \prin^k_\circ \cap P_{0, +}$ and let $\beta = -y(\eta)$ where $\la = y(\nu - (p/u)\eta) -\rho$. Then $\left<\beta, \al^\vee \right> \equiv 0 \pmod{u}$ for $\al \in \D_0$, and $\left<\beta, \al^\vee \right> \not\equiv 0 \pmod{u}$ for $\al \notin \D_0$.
\end{lemma}

\begin{proof}
For the first claim, evidently $\left<\la, \al^\vee\right> = \left<y(\nu)-\rho, \al^\vee\right> + (p/u)\left<\beta, \al^\vee\right>$ is non integral, unless $\left<\beta, \al^\vee \right> \equiv 0 \pmod{u}$, as required. For the second claim, we note that if $\left<\beta, \al^\vee \right> \equiv 0 \pmod{u}$ held for some root $\al \notin \D_0$, without loss of generality positive, then either $\left<\la, \al^\vee \right> \in \Z_{<0}$ contradicting the definition of admissible weight, or $\la$ would be integrable with respect to $\al^\vee$. But we would then have $|\D(\la)| > |\D_0|$ which, by {\cite[Theorem 3.6]{AE2023}}, would contradict $\la \in \prin_\circ^k$.
\end{proof}

We conclude this subsection with a few comments about the ``boundary admissible'' case which will occupy most of our attention in the remainder of the paper. A principal admissible level is said to be a boundary level if it is of the form $k = -h^\vee + h^\vee / u$. Since $P_+^{h^\vee, \text{reg}} = \{\rho\}$, the parametrisation of $W(\mbo_u, p/u)$-modules stated in Theorem \ref{thm:rationality.thm} reduces to $\prin^k_\circ / \sim = \check{P}_+^{u} / \wtil{W}_+$.

\subsection{Modular data}\label{subsec:mod.data}

In \cite{Huang-rigidity} Huang has shown that the category of representations of a lisse rational self-dual vertex algebra (under additional minor hypotheses) carries the structure of a modular tensor category \cite{MS89, EGNO, BakKir}. For a general modular tensor category, one has matrices $T$ (expressing the twists of simple objects) and $S$ (expressing the evaluation of the Hopf link labeled by pairs of simple objects), and these satisfy relations which are encoded in the definition of ``modular data''. More precisely the Grothendieck group of a modular tensor category acquires a commutative ring structure induced by the tensor product, and the simple objects define a basis relative to which we may consider the structure constants of the product. These are known as the fusion rules $N_{ij}^k$, i.e., $[i] * [j] = \sum_{[k]} N^{k}_{ij} [k]$, where $i, j, k$ index simple objects. In particular for a vertex algebra $V$ as above, we denote by $\CF(V)$ the Grothendieck group of its modular tensor category of representations. We now recall the notion of modular datum from \cite[Definition 2.1]{GM}:
\begin{defn}
A modular datum consists of a finite index set $\Phi$, two distinguished indices $\vac \in \Phi$ (the vacuum index) and $\circ \in \Phi$ (the minimal index\footnote{The name ``minimal index'' is not standard, but we adopt it because in the case of the module category of a vertex algebra it corresponds to the module which has minimal conformal dimension}), and a pair of square matrices $S$ and $T$ (whose indices run over $\Phi$). The axioms are that
\begin{itemize}
\item $S$ is unitary, $T$ is diagonal and both are symmetric,

\item $S_{\vac, i} \in \R^\times$ and $S_{\circ, i} \in \R_{> 0}$ for each $i \in \Phi$,

\item $S^2 = (ST)^3$,

\item The fusion rules
\[
N_{ij}^k = \sum_{x \in \Phi} \frac{S_{i, x} S_{j, x} \ov{S}_{k, x} }{S_{\vac, x}}
\]
all lie in $\Z_{\geq 0}$.
\end{itemize}
\end{defn}

If $(S, T)$ is a modular datum, then $S^2$ is a permutation matrix, which we denote $C$, furthermore $C^2 = I$. So we denote by $i^\vee \in \Phi$ the unique index corresponding to $i \in \Phi$ under $C$. The matrix $S$ has all real entries if and only if $C = I$.

We briefly recall the notion of asymptotic data. If a series $f(\tau) = q^\lambda \sum_{n=0}^\infty f_n q^n$ in real powers of $q = e^{2\pi i \tau}$ with non negative real coefficients $f_n$ converges for $\operatorname{Im}(\tau) > 0$, and satisfies the asymptotic
\[
f(iT) \sim A e^{\frac{\pi g}{12T}}, \quad T \rightarrow 0_+
\]
(where $T \in \R_{>0}$) for some $g >0$ and $A > 0$, then we say $f(q)$ has asymptotic growth $g$ and asymptotic dimension $A$. We refer to the pair $(A, g)$ as the asymptotic datum of $f(\tau)$. The character $\chi_M(\tau) = \tr_{M} q^{L_0-c/24}$ of an irreducible module $M$ over a lisse rational vertex algebra $V$ admits an asymptotic datum, this being a well known consequence of Zhu's modularity theorem {\cite{Zhu96}}, see \cite{DJX} and {\cite[Proposition 2.2]{AEM2024}}.

The minimal eigenvalue of $L_0$ on an irreducible $V$-module $M$ is known as its conformal dimension, denoted $h_M$. Now let us suppose the minimal (i.e., most negative) conformal dimension is attained on a unique irreducible $V$-module $M_\circ$. This is believed to be true in general, and is verified in the cases treated in this paper (see Section \ref{subsec:T-matrices}). Then for each irreducible $V$-module $M_i$ we have
\begin{align}\label{eq:qdim.formula}
\frac{S_{i, \circ}}{S_{\vac, \circ}} = \frac{A_i}{A_1} \in \R_{>0}.
\end{align}
This quantity is called the quantum dimension of $M_i$.

More generally a ``dimension'' on a modular tensor category (or on any fusion category) is a ring homomorphism from its Grothendieck ring to $\C$. There is a unique dimension whose values on the simple objects are positive real numbers. This can be taken as a definition of the Frobenius-Perron dimension. More explicitly, though:
\begin{defn}
Let us denote by $N(i)$ the matrix $N(i)_{jk} = N_{ij}^k$. The eigenvalues of $N(i)$ are real and at least one is positive. The Frobenius-Perron dimension $\FPdim(i)$ of $i \in \Phi$ is by definition the largest eigenvalue of $N(i)$.
\end{defn}
The Frobenius-Perron dimension is also given by the formula $\FPdim(i) = {S_{i, \circ}}/{S_{\vac, \circ}}$ \cite{GM}.

The positivity entailed by \eqref{eq:qdim.formula} is quite powerful, having in particular the following easy consequence will save us some effort when it comes to pin down certain signs appearing in Weyl group sums.
\begin{lemma}\label{lem:S.untwistable}
Let $(\Phi, S, T)$ and $(\Phi, \widetilde{S}, T)$ be two modular data on the same set $\Phi$, such that
\[
\widetilde{S}_{ij} = \varepsilon_i \varepsilon_j S_{ij},
\]
where $\varepsilon_{i} \in \{+1, -1\}$ for each $i \in \Phi$. Then $\varepsilon_i = +1$ for all $i \in \Phi$ and $\widetilde{S} = S$.
\end{lemma}

\begin{proof}
Since for both modular data, the quantum dimensions must take positive real values, we have $\varepsilon_i / \varepsilon_\vac \in \R_{>0}$ for each $i \in \Phi$. Thus all the $\varepsilon_i$ are equal, and in either case we recover $\widetilde{S} = S$.
\end{proof}

\section{General results}

In this section we gather some results on characters and their modular transformations, for exceptional $W$-algebras. Although our chief application will be in types $A$ and $D$, the results of this section apply in other types also.

\subsection{A product formula for $q$-characters}

The following result is a consequence of {\cite[Formula (11)]{KW.boundary}}.
\begin{prop}\label{prop:H.la.char.thm}
Let $k = -h^\vee + p/u$ be an admissible level and $\la = y(\rho - (p/u)\eta)-\rho \in \prin_\circ^k \cap P_{0, +}$ a replete $\g_0$-integrable admissible weight. Then
\begin{align}\label{eq:H.la.char.thm}
\widetilde{\chi}_{H^0_{f, -}(L_k(\la))}(\tau) = \prod_{n \geq 1} \frac{(1-q^{un})^\ell}{(1-q^n)^{\dim(\g_0)}} \prod_{\substack{n \in \Z, \al \in \D \\ un + (\eta, \al) > 0}} (1-q^{un + (\eta, \al)}).
\end{align}
\end{prop}

A pleasant feature of formula \eqref{eq:H.la.char.thm} is its dependence on $\la$ only through $\eta$ (and not $y$).

\begin{proof}[Proof sketch]
The Weyl-type character formula of Kac and Wakimoto \cite{KW88} gives
\[
\chi_{L_k(\la)}(\tau, \xi) = \frac{D(u y^{-1} t_{-\beta/u}(\tau, \xi/u, t/u^2) )}{D(\tau, \xi, t)},
\]
where $D$ is the Weyl denominator for the affine root system $\what\D$, and $\what{\h}^*$ is coordinatised as $(\tau, \xi, t) = 2\pi i \left( -\tau \La_0 + \sum_i \xi_i \varpi^\vee_i + t \delta \right)$. From now on we ignore the term $t\delta$. In these coordinates
\[
D(\tau, \xi) = \prod_{n \geq 1} \left[ (1-q^{n})^{\ell} \prod_{\al \in \D_+} (1 - e^{-2\pi i \sum \xi_i (\varpi^\vee_i | \al)} q^{n-1}) (1 - e^{2\pi i \sum \xi_i (\varpi^\vee_i | \al)} q^{n}) \right],
\]
and the numerator is transformed to $D(u\tau, y^{-1}(z) - \tau \eta)$.

The effect of the reduction $H_{f, -}^0(-)$ at the level of characters is to multiply $\chi_{L_k(\la)}$ by the fermion character
\[
\chi_{F}(\tau, z) = \prod_{n \geq 1} \prod_{\al \in \D_{>0}} (1 - e^{-2\pi i \sum \xi_i (\varpi^\vee_i | \al)} q^{n-1}) (1 - e^{2\pi i \sum \xi_i (\varpi^\vee_i | \al)} q^{n}),
\]
and take the limit $\xi \rightarrow 0$.

In the limit, zeros appear in both the numerator and denominator. In the denominator they are the terms
\[
1 - e^{-2\pi i \sum \xi_i (\varpi^\vee_i | \al)} q^{n-1} \quad \text{for $\al \in \D_{0, +}$ and $n=1$}.
\]
In the numerator they are the terms
\begin{align*}
1 - e^{-2\pi i \sum \xi_i (\varpi^\vee_i | y(\al))}  q^{u(n-1) + (\eta, \al)}\quad &\text{for $\al \in \D_{+}$ such that $(\eta, \al) = 0$ and $n=1$}, \\
\text{and} \quad 1 - e^{2\pi i \sum \xi_i (\varpi^\vee_i | y(\al))}  q^{un - (\eta, \al)}\quad &\text{for $\al \in \D_{+}$ such that $(\eta, \al) = u$ and $n=1$}.
\end{align*}
They cancel when $\D_{0, +}$ coincides with the set
\[
\{y(\al) \mid \text{$\al \in \D_+$ and $(\eta, \al) \equiv 0 \pmod{u}$} \}.
\]
If we choose $y$ so that $\la$ is $\g_0$-integrable, then this condition is satisfied by Lemma \ref{lem:replete.integrability.general}. With a little additional manipulation, we arrive at the stated result.

\end{proof}
We now provide a couple of examples and applications of Proposition \ref{prop:H.la.char.thm}.

First we verify the isomorphism $W(\sll_3[2, 1], 3/2) \cong \C \vac$ at the level of characters. We choose the good even grading on $\sll_3$ associated with the left adjusted pyramid (see Section \ref{sec:type.A} below for a review of the combinatorics of pyramids and more details on the parametrisation of admissible weights, including the definition of $\check{P}_{+, f}^u$). This gives $\D_{0, +} = \{\al_1\}$ and $\dim(\g_0) = 4$. We have $P_+^{3, \text{reg}} = \{\rho\}$ and $\check{P}_{+, f}^2$ consists of a single $\wtil{W}/\what{W}$-orbit containing the weights $\varpi_1$, $\varpi_2$ and $\rho$. No matter which representative $\eta \in \check{P}_{+, f}^2$ we choose, we obtain the expected result
\begin{align*}
\widetilde{\chi}_{H^0_{f, -}(L_k(\la))}(\tau) = \prod_{n \geq 1} \frac{(1-q^{2n})^2}{(1-q^n)^{4}} \prod_{n \geq 1} (1-q^{2n}) (1-q^{2n - 1})^2 (1-q^{2n + 1})^2 (1-q^{2n})
= 1.
\end{align*}

\subsection{Application to product formulas for Virasoro characters}\label{subsec:vir.prods}

In this section we consider a class of boundary admissible $W$-algebras in type $D$ and uncover a conceptual explanation of some observed product formulas for $q$-characters of modules over the Virasoro algebra.

We denote by $L(c, h)$ the irreducible highest weight representation of the Virasoro Lie algebra of central charge $c$ and in which $L_0$ acts on the highest weight vector by the constant $h$. The module $L(c, 0)$ carries a vertex algebra structure (simple quotient of the universal Virasoro vertex algebra $\Vir^c$). The central charges $c_{p, u} = 1-6(p-u)^2/pu$, for $p, u \geq 2$ coprime, are special and the corresponding simple vertex algebras $\Vir_{p, u} =L(c_{p, u}, 0)$ are called minimal models.
\begin{prop}
Let $u \geq 1$ be odd and not divisible by $3$. Set $n = (u-1)/2$ and consider the principal $W$-algebra $B_u = W(D_n, (u-3)/u)$. Then
\begin{enumerate}
\item there exists an inclusion $L(c, 0) \subset B_u$ of $\Vir$-modules, where the central charge is $c = c_{3, 2u} = 13 - 4u - 9/u$, and furthermore $B_u$ decomposes into irreducible $\Vir$-modules as
\[
B_u = L(c, 0) \oplus L(c, n),
\]

\item the $q$-character of $B_u$ is given by the infinite product
\begin{align*}
\chi_{B_u}(q) = \frac{(q^{u}; q^{u})_\infty}{(q;q)_{\infty}} \cdot \frac{(q; q^{u})_\infty (q^{u-1}; q^{u})_\infty}{(q^{(u-1)/2}; q^{u})_\infty(q^{(u+1)/2}; q^{u})_\infty}
\end{align*}
\textup{(}where $(a;q)_\infty = \prod_{n=0}^\infty (1-aq^k)$ is the $q$-Pochhammer symbol\textup{)},

\item and, if for $m \geq 0$ even, we set $2N = mu+u-1$ then the exceptional $W$-algebra $W(\mbo_u, (2N-2)/u)$ is isomorphic to $B_u$.
\end{enumerate}
\end{prop}

Nilpotent orbits in $D_N$ are parametrised by certain partitions of $2N$. The orbit $\mbo_u$ in the theorem statement corresponds to the partition $[u^m, u-2, 1]$. We refer the reader to \cite{CMcG} for details.

\begin{proof}
We use asymptotic data. It follows from {\cite[Proposition 4]{KW88}} that a highest weight $\Vir$-module of central charge $c_{p, u}$ and asymptotic growth strictly less than $1$, must be a finite direct sum of irreducible representations of the rational vertex algebra $\Vir_{p, u}$. The asymptotic dimensions of the irreducible $\Vir_{p, u}$-modules are also given in that article explicitly.

Having fixed $u$, it is readily checked that $\Vir_{3, 2u}$ and all the $W$-algebras cited in the theorem statement have the same central charge of $c = 13 - 4u - 9/u$ {\cite[Theorem 2.2]{KRW03}} and the same asymptotic growth $g = 1-1/u$ {\cite[Theorem 2.16]{KW2009}}. So the $W$-algebras are finite extensions of $\Vir_{3, 2u}$. By explicit computation the asymptotic dimension of the $W$-algebra is twice that of $\Vir_{3, 2u}$, and at the same time the asymptotic dimensions of the irreducible $\Vir_{3, 2u}$-modules are seen to be greater than or equal to that of $L(c, 0)$ with equality only for $L(c, n)$. It follows that the decomposition of $B_u$ is $L(c, 0) \oplus L(c, n)$. So the first claim is established.

Let $\g = \bigoplus_{j \in \Z} \g_j$ be the decomposition of $\g = D_n$ relative to the principal nilpotent element. With an explicit description of the $D_n$ root system in hand, and elementary counting arguments, one finds that $\dim(\g_j) = n-\lfloor j/2 \rfloor$ for $1 \leq j \leq n-1$ and $\dim(\g_j) = n-1-\lfloor j/2 \rfloor$ for $n \leq j \leq 2n-3$. Knowing these dimensions, the second claim follows from Proposition \ref{prop:H.la.char.thm}.

The third claim follows from the theory of simple current extensions. In fact the irreducible $L(c, 0)$-module $L(c, n)$ is a simple current, and so the direct sum $L(c, 0) \oplus L(c, n)$ possesses a unique vertex algebra structure {\cite[Section 5]{DM-ceff}} {\cite[Section 2]{Y}}.

\end{proof}

\begin{rem}
The fact that the $q$-character of $L(c, 0) \oplus L(c, n)$, for $c = c_{3, 2u}$ and $n = (u-1)/2$, possesses a particularly nice product expression was observed in \cite{BF2000}. This observation was further explored in \cite{FFW}, where a probable link with the $W_3$-algebra $W(\sll_3[3], 3/7)$ (there denoted $M(3, 7)_3$) was noted in connection with the specific case of $\Vir_{3, 14}$. The uniform conceptual explanation of these product formulas via the interpretation as boundary exceptional $W$-algebras in type $D$ seems to be new.
\end{rem}

\subsection{A formula for the $S$-matrix}\label{subsec:S-mat.genfla}

Let $k = -h^\vee + p/u$ be a principal admissible level, and fix a complete set $\laset \subset P_{0, +}$ of $W \circ (-)$-representatives of elements of $\prin^k_\circ$, written as
\begin{align*}
\la = y(\nu - (p/u) \eta) - \rho.
\end{align*}
We also denote $\beta = -y(\eta)$ associated to $\la \in \laset$ and we denote by $\beset$ the set of these weights.
\begin{rem}
Soon we will restrict attention to the boundary case $p = h^\vee$, this means that $\nu = \rho$ for all $\la \in \laset$ since $P_{+}^{h^\vee, \text{reg}} = \{\rho\}$.
\end{rem}

Consider the character $\chi_\la(\tau, z)$ of the admissible module $L_k(\la)$. As $\la$ runs over $\prin^k$, these characters are linearly independent, and form a representation of $SL_2(\Z)$, so that \cite{KW89}
\[
\chi_\la\left(-\frac{1}{\tau}, \frac{z}{\tau}\right) = e^{\pi i \frac{p}{u} \frac{(z, z)}{\tau}} \sum_{\la' \in \prin^k} a(\la, \la') \chi_{\la'}(\tau, z),
\]
where the coefficients $a(\la, \la')$ are given explicitly by
\begin{align}\label{eq:aB}
a(\la, \la') = \frac{i^{|\D_+|}}{\sqrt{|P / puQ|}} \varepsilon(y) \varepsilon(y') e^{-2\pi i [(\nu, \beta') + (\nu', \beta)]} e^{-2\pi i \frac{p}{u} (\beta, \beta')} \sum_{w \in W} \varepsilon(w) e^{-2\pi i \frac{u}{p} (w(\nu), \nu')}.
\end{align}

We denote by $W(\Gamma) \subset W$ the set
\[
W(\Gamma) = \{w \in W \mid w(\D_{0, +}) \subset \D_+\}.
\]
Note that $W(\Gamma) \subset W$ is a subset and not in general a subgroup.

The following theorem is a mild generalisation of {\cite[Theorem 10.4]{AE2023}} and its proof is the same.
\begin{prop}\label{prop:S-matrix}
Assume each element of $\prin^k_\circ / \sim$ has a representative in $P_{0, +}$ and let $\laset$ be a fixed set of such representatives as above. Then the $S$-matrix of the exceptional $W$-algebra $W(\mbo_u, p/u)$ is given by
\begin{align}\label{eq:S.orig}
S_{\la, \la'} = i^{|\D_{>0}|} \sum_{w \in W(\Gamma)} a(\la, w \circ \la') \prod_{\al \in \D_{0, +}} \frac{(w(\al), \xi)}{(\al, \xi)},
\end{align}
where $\xi \in \h^*$ is an auxiliary parameter chosen so $(\al, \xi) \neq 0$ for all $\al \in \D_{0, +}$.
\end{prop}

Now let us restrict to $p=h^\vee$. The explicit formula for $a$ simplifies substantially. 
\begin{prop}\label{prop:S-matrix-boundary}
For a boundary admissible level, satisfying the same hypotheses as in Proposition \ref{prop:S-matrix}, the $S$-matrix takes the form
\begin{align}\label{eq:S.orig.firstsimp}
S_{\la, \la'} = C(\la, \la') \sum_{w \in W(\Gamma)} \varepsilon(w) e^{-2\pi i \frac{h^\vee}{u} (\beta, w(\beta'))} \prod_{\al \in \D_{0, +}} \frac{(w(\al), \xi)}{(\al, \xi)}
\end{align}
where $C(\la, \la') = C_0 \varepsilon(\la) \varepsilon(\la')$ for some constant $C_0$ and some function $\varepsilon : \widehat{X} \rightarrow \{+1, -1\}$.
\end{prop}

\begin{proof}
For $p=h^\vee$ we have $P_+^{p, \text{reg}} = \{\rho\}$. Therefore the summation over $W$ appearing in \eqref{eq:aB} has $\nu = \nu'= \rho$, and so takes the same value for all pairs $(\la, \la')$. The term
\[
e^{-2\pi i [(\nu, \beta') + (\nu', \beta)]} = e^{-2\pi i (\rho, \beta + \beta')}
\]
takes values $\pm 1$ because $\beta, \beta \in P$ and $\rho \in \frac{1}{2}Q$. These observations give the desired conclusion.
\end{proof}

\begin{rem}
Of course it might be desirable to have the function $C(\la, \la')$ in a more explicit form, and this can be done, but the details will depend somewhat on the particular application. Under the action of $\widetilde{W} / \widehat{W}$ on $P_{+}^u$, each orbit meets the root lattice $Q$. If we choose the representatives $\la \in \laset$ so that $\eta \in Q$ then we have $\varepsilon(\la) = \varepsilon(y)$. For some Lie algebras, such as $\sll_n$ for $n$ odd, we have $\rho \in Q$ and so this simplification occurs whether or not we make special choices of $\eta$.
\end{rem}

\subsection{Modular data of $W(E_8(a_1), 31/29)$}\label{subsec:W29Smat}

In this section we present a computation of modular data for a subregular $W$-algebra of type $E_8$. We take $u=29$ so $\mbo_u$ is the subregular orbit. In fact the arguments work for other subregular denominators, but in these cases the modular data is more easily determined via isomorphisms with better understood vertex algebras \cite{vE-N}. So for the rest of this section we let $\fing = E_8$ and we denote by $\alstar$ the simple root associated with the trivalent node of the Dynkin diagram of $\fing$. A subregular nilpotent element $f \in \fing$ can be chosen so that the associated Dynkin grading, which is even, has $\D_{0, +} = \{\alstar\}$ and $\al \in \D_1$ for all $\al \in \Pi \backslash \{\alstar\}$.

The $S$-matrix is given by the formula \eqref{eq:S.orig}, which involves a sum over the set  $W(\Gamma)$. In this case the cardinality of $W(\Gamma)$ is half that of the whole of $W$, which in turn is $696729600$. So it is not practical to evaluate the sum directly.

Let us consider the root system $\D_{E_7}$, with simple roots labeled as indicated in the figure below. We put $\al^\circ = -(2\al_1+2\al_2+3\al_3+4\al_4+3\al_5+2\al_6+\al_7)$, then it turns out the root subsystem spanned by $\Pi^\bullet = \{\al^\circ, \al_1, \al_3, \al_4, \al_5, \al_6, \al_7\}$ is of type $A_7$, and $\Pi^\bullet$ is a set of simple roots for this root subsystem.
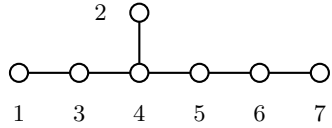
\begin{center}
  \begin{tikzpicture}[scale=.4]
    \draw (-1,1) node[anchor=east]  {};
    \foreach \x in {0,...,5}
    \draw[thick,xshift=\x cm] (\x cm,0) circle (3 mm);
    \foreach \y in {0,...,4}
    \draw[thick,xshift=\y cm] (\y cm,0) ++(.3 cm, 0) -- +(14 mm,0);
    \draw[thick] (4 cm,2 cm) circle (3 mm);
    \draw[thick] (4 cm, 3mm) -- +(0, 1.4 cm);
    \draw[xshift=0 cm,thick] node[below=8pt,fill=white] {\scriptsize{$1$}};
    \draw[xshift=2 cm,thick] node[below=8pt,fill=white] {\scriptsize{$3$}};
    \draw[xshift=4 cm,thick] node[below=8pt,fill=white] {\scriptsize{$4$}};
    \draw[xshift=6 cm,thick] node[below=8pt,fill=white] {\scriptsize{$5$}};
    \draw[xshift=8 cm,thick] node[below=8pt,fill=white] {\scriptsize{$6$}};
    \draw[xshift=10 cm,thick] node[below=8pt,fill=white] {\scriptsize{$7$}};
    \draw[xshift=4 cm, yshift=2cm, thick] node[left=8pt,fill=white] {\scriptsize{$2$}};
  \end{tikzpicture}
\end{center}
Now we consider the root system $\D$ of $\g = E_8$. The root subsystem $\D^\perp = \{\al \in \D | (\al, \alstar) = 0\}$ is of type $E_7$. The subgroup $W^\perp \subset W$ generated by reflections in elements of $\D^\perp$ coincides with the stabiliser of $\alstar$, and so has index $240$. Let us denote by $\D^\bullet$ the root subsystem $\Z\Pi \cap \D$ of type $A_7$ embedded in $\D^\perp$ as above. This induces an embedding of $W^\bullet \cong W_{A_7}$ into $W^\perp$ of index $72$.


Consider the presentation of the root system of type $A_{n-1}$ as the subset
\[
\D_{A_{n-1}} = \{(x_1,\ldots, x_{n}) \in \Z^{n} \mid \text{$\sum_i x_i = 0$ and $(x, x)=2$}\}
\]
of $\Z^{n}$ with its standard scalar product. The Weyl group of $\D_{A_{n-1}}$ is isomorphic to $S_{n}$ and acts on vectors of $\R \D_{A_{n-1}}$ by permuting their entries. One now has
\[
\sum_{w \in S_n} \varepsilon(w) e^{-\frac{2\pi i}{u} (\la, w(\mu))} = \det{M^{(u)}},
\]
where $M^{(u)}$ is the $n \times n$ matrix defined by $M_{ij}^{(u)} = e^{-\frac{2\pi i}{u} \la_i \mu_j}$. The determinant can be computed much more efficiently than by direct evaluation of the sum on the left hand side. We learned this trick from T. Gannon.

To apply this observation to the present case we rewrite the sum over $W(\Gamma)$ in
\begin{align}\label{eq:Ssum.E8.sec}
\sum_{w \in W(\Gamma)} \varepsilon(w) e^{-2\pi i \frac{h^\vee}{u} (\beta, w(\beta'))} \prod_{\al \in \D_{0, +}} \frac{(w(\al), \xi)}{(\al, \xi)}
\end{align}
as a double sum over coset representatives $[\sigma] \in W / W^\perp$ and $[\tau] \in W^\perp / W^\bullet$ of a sum over $W^\bullet$. The sum over $W^\bullet \cong W_{A_7}$ may be evaluated as a determinant as described above. It is now convenient to substitute $\xi = \alstar$ since then $\left(w(\alstar), \xi\right)$ takes on values $-2, -1, 0, 1, 2$, and depends only on $[\sigma]$. We may restrict the sum on $[\sigma]$ to those $57$ cosets for which $\left(\sigma(\alstar), \alstar\right)$ equals $1$ or $2$. The sum \eqref{eq:Ssum.E8.sec} becomes
\begin{align*}
\frac{1}{2} \sum_{\stackrel{[\sigma] \in W_{E_8} / W_{E_7}}{\left(\sigma(\alstar), \alstar\right) \in \{1, 2\}}} {\left(\sigma(\alstar), \alstar\right)} \sum_{[\tau] \in W_{E_7} / W_{A_7}} \varepsilon(\sigma\tau) \det{M^{(u)}(\tau^{-1}\sigma^{-1}(\beta), \beta')}.
\end{align*}
The computation of an entry of the $S$-matrix is thus reduced to the computation of $57 \times 72$ determinants of $8 \times 8$ matrices, which is feasible.

Now we present the modular data thus obtained in the specific case $u=29$. This vertex algebra has central charge $c = -5350/29$ and has $44$ irreducible modules,  constructed as reductions $H_{-}^0(L_k(\lambda))$ of $L_k(\la)$ where $\lambda = y(\rho - (31/29)\eta) - \rho$ and $\eta$ runs over the following list:
{\small
\begin{alignat*}{11}
&0\!:~ & (1,1,1,1,1,1,1,1) &\quad
&1\!:~ & (2,1,1,1,1,1,0,1) &\quad
&2\!:~ & (1,2,1,1,1,0,1,1) &\quad
&3\!:~ & (3,1,1,1,0,1,1,1) \\
&4\!:~ & (1,1,2,1,0,1,1,1) &\quad
&5\!:~ & (2,2,1,0,1,1,1,1) &\quad
&6\!:~ & (1,2,0,1,1,1,1,1) &\quad
&7\!:~ & (2,0,1,1,1,1,1,1) \\
&8\!:~ & (1,1,1,0,2,1,1,1) &\quad
&9\!:~ & (1,1,1,1,0,2,1,1) &\quad
&10\!:~ & (1,1,1,1,1,0,2,1) &\quad
&11\!:~ & (2,1,1,0,1,1,2,1) \\
&12\!:~ & (1,1,0,1,1,1,2,1) &\quad
&13\!:~ & (1,1,1,1,1,1,0,2) &\quad
&14\!:~ & (2,1,1,1,0,1,1,2) &\quad
&15\!:~ & (1,2,1,0,1,1,1,2) \\
&16\!:~ & (1,0,1,1,1,1,1,2) &\quad
&17\!:~ & (1,1,1,0,1,1,2,2) &\quad
&18\!:~ & (1,1,1,1,0,1,1,3) &\quad
&19\!:~ & (1,1,1,1,1,1,1,0) \\
&20\!:~ & (2,1,1,1,1,0,1,1) &\quad
&21\!:~ & (1,2,1,1,0,1,1,1) &\quad
&22\!:~ & (3,1,1,0,1,1,1,1) &\quad
&23\!:~ & (1,1,2,0,1,1,1,1) \\
&24\!:~ & (2,1,0,1,1,1,1,1) &\quad
&25\!:~ & (0,1,1,1,1,1,1,1) &\quad
&26\!:~ & (1,1,1,0,1,2,1,1) &\quad
&27\!:~ & (1,1,1,1,0,1,2,1) \\
&28\!:~ & (1,1,1,1,1,0,1,2) &\quad
&29\!:~ & (2,1,1,0,1,1,1,2) &\quad
&30\!:~ & (1,1,0,1,1,1,1,2) &\quad
&31\!:~ & (1,1,1,0,1,1,1,3) \\
&32\!:~ & (1,1,1,1,1,1,0,1) &\quad
&33\!:~ & (2,1,1,1,0,1,1,1) &\quad
&34\!:~ & (1,2,1,0,1,1,1,1) &\quad
&35\!:~ & (1,0,1,1,1,1,1,1) \\
&36\!:~ & (1,1,1,0,1,1,2,1) &\quad
&37\!:~ & (1,1,1,1,0,1,1,2) &\quad
&38\!:~ & (1,1,1,1,1,0,1,1) &\quad
&39\!:~ & (2,1,1,0,1,1,1,1) \\
&40\!:~ & (1,1,0,1,1,1,1,1) &\quad
&41\!:~ & (1,1,1,0,1,1,1,2) &\quad
&42\!:~ & (1,1,1,1,0,1,1,1) &\quad
&43\!:~ & (1,1,1,0,1,1,1,1)
\end{alignat*}
}
the weights being expressed relative to the basis of fundamental weights, with simple roots indexed according to the following labeled Dynkin diagram 
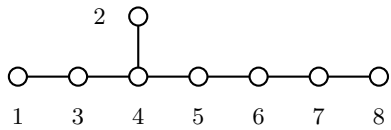
\begin{center}
  \begin{tikzpicture}[scale=.4]
    \draw (-1,1) node[anchor=east]  {};
    \foreach \x in {0,...,6}
    \draw[thick,xshift=\x cm] (\x cm,0) circle (3 mm);
    \foreach \y in {0,...,5}
    \draw[thick,xshift=\y cm] (\y cm,0) ++(.3 cm, 0) -- +(14 mm,0);
    \draw[thick] (4 cm,2 cm) circle (3 mm);
    \draw[thick] (4 cm, 3mm) -- +(0, 1.4 cm);
    \draw[xshift=0 cm,thick] node[below=8pt,fill=white] {\scriptsize{$1$}};
    \draw[xshift=2 cm,thick] node[below=8pt,fill=white] {\scriptsize{$3$}};
    \draw[xshift=4 cm,thick] node[below=8pt,fill=white] {\scriptsize{$4$}};
    \draw[xshift=6 cm,thick] node[below=8pt,fill=white] {\scriptsize{$5$}};
    \draw[xshift=8 cm,thick] node[below=8pt,fill=white] {\scriptsize{$6$}};
    \draw[xshift=10 cm,thick] node[below=8pt,fill=white] {\scriptsize{$7$}};
    \draw[xshift=12 cm,thick] node[below=8pt,fill=white] {\scriptsize{$8$}};    
    \draw[xshift=4 cm, yshift=2cm, thick] node[left=8pt,fill=white] {\scriptsize{$2$}};
  \end{tikzpicture}
\end{center}
Choices of $y \in W$ so that $\beta = -y(\eta)$ be $\g_0$-integrable can be given, but we do not record them here. The conformal dimension of $H_{-}^0(L_k(\lambda))$ is given by the formula in {\cite[Remark 2.2]{KRW03}}. For instance the module indexed $0$ above has conformal dimension $-195/29$, and the conformal dimensions (multiplied by $-29$) of all the irreducible modules are recorded in the following table, which is hopefully self-explanatory.
{\small
\begin{alignat*}{34}
&0\!:&195 &\quad &1\!:&186 &\quad &2\!:&196 &\quad &3\!:&169 &\quad &4\!:&202 &\quad &5\!:&206 &\quad &6\!:&210 &\quad &7\!:&208 &\quad &8\!:&214 &\quad &9\!:&216 &\quad &10\!:&218 \\
&11\!:&197 &\quad &12\!:&201 &\quad &13\!:&220 &\quad &14\!:&203 &\quad &15\!:&209 &\quad &16\!:&211 &\quad &17\!:&169 &\quad &18\!:&175 &\quad &19\!:&222 &\quad &20\!:&209 &\quad &21\!:&217 \\
&22\!:&188 &\quad &23\!:&221 &\quad &24\!:&223 &\quad &25\!:&225 &\quad &26\!:&204 &\quad &27\!:&208 &\quad &28\!:&212 &\quad &29\!:&191 &\quad &30\!:&195 &\quad &31\!:&132 &\quad &32\!:&216 \\
&33\!:&199 &\quad &34\!:&205 &\quad &35\!:&207 &\quad &36\!:&165 &\quad &37\!:&171 &\quad &38\!:&177 &\quad &39\!:&156 &\quad &40\!:&160 &\quad &41\!:&97 &\quad &42\!:&105 &\quad &43\!:&0
\end{alignat*}
}
The vacuum module is labeled $43$, and the minimal conformal dimension of $-225/29$ is achieved by the module labeled $25$.

The $T$-matrix is the diagonal matrix with entries $e^{2\pi i (h - c/24)}$ as $h$ ranges over the conformal dimensions of the irreducible modules. The entries of $T$ lie in the cyclotomic field $\Q(\zeta_{N})$ where $N = 2^2 \cdot 3 \cdot 29$. In the rest of this section we describe the $S$-matrix, computed in the previous section. The entries of the $S$-matrix lie in $\Q(\zeta_{N'})$ where $N' = 2^2 \cdot 29$ and in general are very complicated, so it will not do to reproduce them explicitly here. Luckily modular data exhibit a high degree of symmetry and below we are able to present a description which allows one to reconstruct $S$ completely.


We paraphrase the following remarkable result from {\cite[Proposition 2.2]{GM}} (see also \cite{CG94} and \cite{CG99}).
\begin{prop}
Let $(\Phi, S, T)$ be the modular datum associated with a modular tensor category, and suppose the matrices $S$ and $T$ are defined over $\Q(\zeta_N)$ as above. Let $\sigma \in \Gal(\Q(\zeta_N)/\Q)$ be the automorphism characterised by $\sigma(\zeta_N) = \zeta_N^\ell$, let $\ell^{-1}$ denote the multiplicative inverse of $\ell$ modulo $N$, and let
\[
G = C ST^{\ell^{-1}} ST^\ell ST^{\ell^{-1}}.
\]
Then $G$ is a signed permutation matrix, and thus defines a permutation $\Sigma$ of $\Phi$ by $G_{i, \Sigma(i)} \neq 0$, and a function $\varepsilon : \Phi \rightarrow {\pm 1}$ by $G_{i, \Sigma(i)} = \varepsilon(i)$. Furthermore
\[
S_{\Sigma(i), j} = \varepsilon(i) \sigma(S_{i, j}) \quad \text{and} \quad S_{i, \Sigma(j)} = \varepsilon(j) \sigma(S_{i, j}).
\]
\end{prop}
In the present case we have $\Gal(\Q(\zeta_N)/\Q) \cong \Z/2 \times \Z/28$.
Let us fix a generator $\sigma$ of the subgroup $\Z/28$ by defining $\sigma(\zeta_N) = \zeta_N^\ell$ where, say, $\ell = 11$. Then the inverse is $\ell^{-1} = 95$. The matrix $G$ is confirmed to be a signed permutation matrix, and the corresponding $(\Sigma, \varepsilon)$ are summarised in the following table
\begin{align*}
&(43 \xrightarrow[-]{} 13 \xrightarrow[-]{} 40 \xrightarrow[-]{} 27 \xrightarrow[+]{} 32 \xrightarrow[+]{} 3 \xrightarrow[+]{} 30 \xrightarrow[+]{} 20 \xrightarrow[-]{} 23 \xrightarrow[+]{} 24 \xrightarrow[-]{} 26 \xrightarrow[+]{} 2 \xrightarrow[+]{} 8 \xrightarrow[-]{} 21 \xrightarrow[-]{}) \\
&(25 \xrightarrow[+]{} 39 \xrightarrow[+]{} 22 \xrightarrow[+]{} 14 \xrightarrow[-]{} 29 \xrightarrow[-]{} 10 \xrightarrow[-]{} 7 \xrightarrow[+]{} 9 \xrightarrow[+]{} 17 \xrightarrow[+]{} 0 \xrightarrow[+]{} 15 \xrightarrow[+]{} 42 \xrightarrow[+]{} 36 \xrightarrow[-]{} 18 \xrightarrow[-]{} ) \\
&(1 \xrightarrow[-]{} 19 \xrightarrow[-]{} 33 \xrightarrow[+]{} 37 \xrightarrow[+]{} 34 \xrightarrow[-]{} 12 \xrightarrow[-]{} 6 \xrightarrow[+]{} 11 \xrightarrow[-]{} 31 \xrightarrow[+]{} 41 \xrightarrow[+]{} 28 \xrightarrow[+]{} 35 \xrightarrow[-]{} 16 \xrightarrow[+]{} 4 \xrightarrow[-]{}) \\
&(5 \xrightarrow[-]{} 38 \xrightarrow[+]{}).
\end{align*}
That is to say, there are four orbits and so the matrix $S$ can be recovered once we specify $S_{ij}$ for fixed representatives $i, j$ of the orbits. So for example $S_{10, 38} = -\sigma^{6}(S_{25, 5})$. Since $S$ is symmetric we need to specify $10$ entries.

Some of these entries are easy to describe, for example
\begin{align*}
29 S_{1, 5} &= \sqrt{29 + 2\sqrt{29}} \cong 6.30637214365513.
\end{align*}
The minimal polynomial of $(29 S_{1,5})^2$ is $x^2 - 58x + 725$. Each of the entries can be reconstructed in $\Q(\zeta_N)$ by specifying its minimal polynomial and a numerical approximation. We now give such data:
\begin{alignat*}{11}
29 S_{43, 43} &\simeq& -5.86534923281479  &\quad& p_1(x)
&\qquad\qquad&
29 S_{25, 1}  &\simeq&  2.17190332799645  &\quad& p_3(x) \\
29 S_{43, 25} &\simeq&  0.00926657067847  &\quad& p_2(x)
&\qquad\qquad&
29 S_{25, 5}  &\simeq&  5.28799697596705  &\quad& p_4(x) \\
29 S_{43, 1}  &\simeq& -1.33861813274628  &\quad& p_3(x)
&\qquad\qquad&
29 S_{1, 1}   &\simeq&  1.58312257487191  &\quad& p_5(x) \\
29 S_{43, 5}  &\simeq& -1.01837516768809  &\quad& p_4(x)
&\qquad\qquad&
29 S_{1, 5}   &\simeq&  6.30637214365513  &\quad& p_6(x) \\
29 S_{25, 25} &\simeq&  7.46186459820243  &\quad& p_1(x)
&\qquad\qquad&
29 S_{5, 5}   &\simeq& -1.01837516768809  &\quad& p_4(x)
\end{alignat*}
Here $p_i(x)$ denotes in each case the minimal polynomial of $(29S_{ij})^2$, and they are:
{\footnotesize
\begin{align*}
p_1(x) &= x^{14} - 290x^{13} + 35525x^{12} - 2401896x^{11} + 98329720x^{10} - 2517652081x^9 + 40259995416x^8 \\
&- 391310993341x^7 + 2210572816012x^6 - 7022812306110x^5 + 12277584079569x^4 \\
&- 11286772515975x^3 + 4954878263930x^2 - 827994062832x + 17249876309 \\
p_2(x) &= x^{14} - 319x^{13} + 43094x^{12} - 3257193x^{11} + 153546416x^{10} - 4761488859x^9 + 99637649928x^8 \\
&- 1416589774627x^7 + 13554181725866x^6 - 84852777409590x^5 + 328679573925815x^4 \\
&- 708019388632942x^3 + 682823267578703x^2 - 200943809123541x + 17249876309 \\
p_3(x) &= x^{14} - 203x^{13} + 17690x^{12} - 877163x^{11} + 27647034x^{10} - 585018943x^9 + 8535296385x^8 \\
&- 86540781317x^7 + 604503876047x^6 - 2828795114335x^5 + 8377778809050x^4 \\
&- 14036640728958x^3 + 10292822746584x^2 - 1172991589012x + 17249876309 \\
p_4(x) &= x^2 - 29x + 29 \\
p_5(x) &= x^{14} - 377x^{13} + 61277x^{12} - 5685160x^{11} + 335089722x^{10} - 13187790803x^9 + 353901096465x^8 \\
&- 6482852772280x^7 + 79663111559462x^6 - 631124813968458x^5 + 3003167570886150x^4 \\
&- 7684363071348972x^3 + 9516426037908824x^2 - 4454590808159851x + 60046819431629 \\
p_6(x) &= x^2 - 58x + 725
\end{align*}
}
It would be very interesting to describe this modular datum more directly in terms of the root system of $E_8$.

\section{Results in type $A$}\label{sec:type.A}

We briefly review the description of nilpotent orbits and auxiliary data in terms of the combinatorics of pyramids \cite{CMcG, EK, BK}.

For a partition $\la = (\la_1, \la_2, \ldots, \la_k)$ of $n$, with $\la_i \geq \la_{i+1}$ for each $i$, a pyramid of $\la$ is an arrangement of $k$ rows of boxes, $\la_i$ boxes in the $i^{\text{th}}$ row (conventionally with the first row at the bottom), in which the centre of each box lies over either the centre or the edge of the box below it, and there are no overhangs.

Nilpotent orbits in $\g = \sll_n$ are well known to be parametrised by partitions of $n$ \cite{CMcG}. We are particularly interested in the nilpotent orbit $\mbo_u$, which is known \cite{Arakawa2015} to correspond to the partition $[u^m, s]$ where $n = mu+s$ and $0 \leq s < u$. We consider only left adjusted pyramids, though the constructions described below are general.

The boxes of the pyramid are labeled with the integers $1, 2, \ldots, n$, starting at the top left, running down each column, and filling the columns sequentially from left to right. For example the labeled pyramid associated with the nilpotent orbit $\mbo_{[8, 3]} \subset \sll_{11}$ is given in the following diagram
\begin{center}
\begin{tikzpicture}
\def\boxsize{0.8cm}
\foreach \n/\col/\row in {
  1/0/1,
  2/0/0,
  3/1/1,
  4/1/0,
  5/2/1,
  6/2/0,
  7/3/0,
  8/4/0,
  9/5/0,
 10/6/0,
 11/7/0
} {
  \draw (\col*\boxsize, \row*\boxsize) rectangle ++(\boxsize, \boxsize);
  \node at ({\col*\boxsize + 0.5*\boxsize}, {\row*\boxsize + 0.5*\boxsize}) {\n};
}
\end{tikzpicture}
\end{center}
Given the labeling, a nilpotent element $f \in \mbo_\la$ can be constructed, as can a good grading $\g = \bigoplus_{j \in \Z} \g_j$ compatible with $f$. The root vectors $E_{ij} \in \sll_n$ are homogeneous with respect to the grading and $E_{ij} \in \g_k$ where $k$ equals the column number of the box labeled $j$ minus that of $i$, for example in the grading determined by the pyramid above we have $E_{5, 8} \in \g_{2}$. We recall the grading element $x^0 \in \h$ from Section \ref{subsec:aff.W} characterised by $[x^0, x] = k x$ for $x \in \g_k$. From the comments above it is clear that $x^0$ (or rather the corresponding weight in $\h^*$) is given by the sum
\begin{align}\label{eq:x0.description}
x^0 = \sum_{i \in B} \varpi_i,
\end{align}
where $B$ is the set of labels along the bottom row of the pyramid.

We denote by $\D_0 \subset \D$ the root system of the subalgebra $\g_0$, as well as $W_0$ the Weyl group of $\D_0$ and $P_{0, +}$ the set of weights dominant integral with respect to $\D_{0, +}$ as in Section \ref{subsec:aff.W}. Clearly
\[
\D_0 = \{\al_{ij} \mid \text{boxes $i$ and $j$ are in the same column}\}.
\]
For the pyramid drawn above, therefore, we see that $\g_0 \cong \sll_2^3 \oplus \C^8$. More generally for a nilpotent element $f \in \mbo_u$ and its good even grading associated with the left adjusted pyramid, we have $\g_0 \cong \sll_{m+1}^s \oplus \sll_m^{u-s} \oplus \C^{u}$.

We recall the centraliser subalgebras $\g^f = \{x \in \g \mid [f, x] = 0\}$ and $\h^f = \h \cap \g^f$, as well as the following objects which will be useful later on.
\begin{defn}
Denote
\[
\D^f = \{\al \in \D \mid \al(\h^f) = 0\},
\]
and
\[
W^f = \left<r_\al \mid \al \in \D^f\right>.
\]
\end{defn}

\begin{rem}
The centraliser $\mfl \subset \g$ of $\h^f$ is a Levi subalgebra whose set of roots is $\D^f$. In type $A$ at least, $f$ is principal in $\mfl$.
\end{rem}
The descriptions of $\D^f$ in terms of the pyramid is, in a sense, dual to that of $\D_0$. Specifically
\[
\D^f = \{\al_{ij} \mid \text{boxes $i$ and $j$ are in the same row}\}.
\]

\subsection{Parametrisation of irreducible modules in type $A$}

The Dynkin diagram of $\what{\sll}_n$ consists of $n$ nodes arranged in a circle, corresponding to the simple roots $\al_0, \al_1, \ldots, \al_{n-1}$. We identify the weight
\[
\what\la = k\La_0 + \sum_i m_i \varpi_i = (k-\sum m_i)\La_0 + \sum_i m_i \La_i
\]
with the labeling of the affine Dynkin diagram by the integers $k-\sum m_i, m_1, \ldots, m_{n-1}$ on nodes $0, 1, \ldots, {n-1}$, respectively. Clearly $\check{P}_+^{u} = P_+^{u}$ corresponds to labelings by non negative integers. From now on, therefore, we will identify elements of $P_+^u$ with such labelings.

We introduce some notation
\begin{align*}
\check{P}_{+, f}^u = \{ \eta \in \check{P}^u_+ \mid \text{$\D(\eta)$ is isometric to $\D_0$} \}.
\end{align*}
We have $\wtil{W}_+ \cong \Z/n$ and its action on $P_+^{u}$ is given by rotation of labeling of the affine Dynkin diagram. The following is a direct translation of {\cite[Lemma 8.6]{AE2023}} into the above notation.
\begin{prop}\label{AE2023-lemma8.6}
Let $n = mu + s$ where $0 \leq s < u$, and let $f \in \mbo_u \subset \sll_n$. Then the labelings corresponding to $\Pvf{u}$ are: 
\begin{itemize}
\item \textup{(}for $m \geq 1$\textup{)} labelings in $P_+^u$ by $0$ and $1$ in which the copies of $0$ are arranged into non consecutive blocks; $s$ blocks of length $m$ and $u-s$ blocks of length $m-1$,

\item \textup{(}for $m=0$\textup{)} labelings in $P_+^u$ in which all entries are positive.
\end{itemize} 
\end{prop}

The following lemma refines Lemma \ref{lem:replete.integrability.general} in type $A$. It was implicit in the proof in \cite{AE2023} of Proposition \ref{AE2023-lemma8.6} above, and it is convenient to state it here for later use. As in Section \ref{subsec:S-mat.genfla}, for $\la \in \prin^k$ we denote $\beta = -y(\eta)$, so that $yt_{-\eta} = t_{\beta}y$ in $\wtil{W}$. Notice that the transformation $\la \mapsto w \circ \la$ corresponds to $\beta \mapsto w(\beta)$. 
\begin{lemma}\label{lem:replete.integrability.type.A}
If $\la \in \prin^k_\circ$ is $\g_0$-integrable then $\left<\beta, \al^\vee \right> = 0$ for $\al \in \D_0$, and $\left<\beta, \al^\vee \right> \neq 0$ for $\al \notin \D_0$.
\end{lemma}

\begin{proof}
Using the $\wtil{W}_+$-action, we have the freedom to choose $\eta$ so that the label on the node $\al_0^\vee$ is nonzero. For such $\eta$ we have $\left<\eta, \theta^\vee\right> \leq u-1$. Since $\beta$ is in the Weyl group orbit of $-\eta$, we have for all $\al \in \D$ that
\[
| \left<\beta, \al^\vee\right> | \leq u-1.
\]
Now the equality $\left<\beta, \al^\vee \right> = 0$ follows immediately from the first part of Lemma \ref{lem:replete.integrability.general}. The rest follows from the second part of Lemma \ref{lem:replete.integrability.general}.
\end{proof}

\begin{rem}
We note that Lemma \ref{lem:replete.integrability.type.A} is not true outside type $A$. For example in the case $\g = E_8$ with $k = -h^\vee + p/u$ where $u = 29$ considered in Section \ref{subsec:W29Smat}, we have $\rho \in \prin_\circ^k$. The nilpotent orbit $\mbo_u$ is the subregular orbit, and we find $\left<\beta, \al_*^\vee\right> = \pm 29$ when $\eta = \rho$. Since the group $\wtil{W}/\what{W}$ is trivial, there is no freedom to choose $\la$ so as to satisfy the conditions in Lemma \ref{lem:replete.integrability.type.A}.
\end{rem}

We now explain how to count the weights in $\check{P}_{+, f}^u$, concluding that these sets are in natural bijection for fixed $u$ and $s$ and varying $m$.

First we consider the case $m \geq 1$. There are $s$ blocks of $0$ of length $m$ (let us call them ``long blocks'') and another $u-s$ blocks of length $m-1$ (``short blocks''), giving a total of $u$ blocks. Since there are also $u$ copies of the label $1$, any pair of consecutive blocks must be separated by exactly one copy of the label $1$. Thus the weight can be reconstructed from the relative positions of the long and short blocks on the circle.

We are thus counting combinatorial objects known as necklaces. A necklace is a circular arrangement of $N$ beads which may each be one of $k$ colours, two arrangements being considered equivalent if they are related by a cyclic permutation. We are interested in the case $k=2$ and $N = u$, and specifically necklaces with $s$ beads of the first colour (long blocks) and $u-s$ of the second colour (short blocks). 

We see already that the sets $\check{P}_{+, f}^u$, for $m \geq 1$, are in natural bijection with each other and with a certain set of necklaces. To see the bijection with the $m=0$ case, we take a necklace and let $a_1, a_2, \ldots, a_s$ denote the lengths of strings of consecutive short blocks. There are $s$ such strings of short blocks because there are $s$ long blocks. Then $(a_1+1, a_2+1, \ldots, a_s+1)$ is the desired weight for $m=0$.
\begin{thm}\label{thm:bijections}
Let $u \geq 2$ and $0 < s < u$ be coprime to $u$. For $m \geq 0$ let $n = mu+s$ and consider the set $\check{P}_{+, f}^u$ for $\sll_n$. These sets, for varying $m$, are in natural bijection with each other, and have cardinality
\[
\frac{1}{u} \binom{u}{s}.
\]
\end{thm}

\begin{proof}
We have already explained the bijections between the sets $\check{P}_{+, f}^u$, and with the set of necklaces. The following formula for the number $N(s_1, s_2)$ of two coloured necklaces with $s_i$ beads of colour $i$, follows from a simple application of Burnside counting {\cite[p. 491]{Stanley.vol2}}:
\[
N(s_1, s_2) = \frac{1}{u} \sum_{d | \gcd(s_1, s_2)} \varphi(d) \frac{(u/d)!}{(s_1/d)!(s_2/d)!}.
\]
Here $u=s_1+s_2$ and $\varphi(n) = \sum_{d \mid n} d$ is the Euler totient function. Since in our case $\gcd(s, u) = 1$ we easily obtain the formula given in the theorem statement.
\end{proof}

\begin{rem}
The following formula has been given in {\cite[Proposition 3.4]{SXY2024}} for the number of irreducible $W(\mbo_u, p/u)$-modules:
\[
\frac{1}{|W_f|}u^{\ell-\ell_f}\prod_{i=1}^{\ell_f} (u-m_i).
\]
Here $\{m_1, \ldots, m_{\ell_f}\}$ are the exponents of the Weyl group $W_f$. In the present case $W_f \cong S_u^m \times S_s$ (the direct product of $m$ copies of the symmetric group $S_u$ and one copy of $S_s$), and the exponents of $S_k$ are the integers $1, 2, \ldots, k-1$. Thus the formula reproduces the cardinality obtained in Theorem \ref{thm:bijections} above by counting necklaces.
\end{rem}

We now work out the bijections between sets $\check{P}_{+, f}^u$ for $u=8$, $s=3$ and $m = 0, 1, 2$ in some detail. 
First we consider the admissible weights $\la = y(\rho - (3/8)\eta) - \rho$ of $W(\sll_3[3], 3/8)$. They are parametrised by certain pairs $(y, \eta)$ where $y \in W$ and $\eta \in P_8$. Since the nilpotent orbit is principal, the weight $\eta$ must be nondegenerate, and there is no restriction on $y$ coming from $\g_0$-integrability. Up to the action of $\wtil{W}_+ \cong \Z/3$ there are $7$ distinct weights $\eta$, namely
\begin{align*}
&(6, 1 \mid 1), \quad
(5, 2 \mid 1), \\
&(4, 3 \mid 1), \quad
(3, 4 \mid 1), \\
&(2, 5 \mid 1), \quad
(4, 2 \mid 2), \\
&(3, 3 \mid 2).
\end{align*}
The $W$-algebra thus has $7$ irreducible modules, given by reductions $H^0_{f, -}(L_k(\la))$.

Next we consider $W(\sll_{11}[8, 3], 11/8)$. We consider the left adjusted pyramid corresponding to the partition $[8, 3]$, from which we read off that $\D_{0, +} = \{\al_1, \al_3, \al_5\}$. The admissible weights $\la = y(\rho - (11/8)\eta) - \rho$ are parametrised by certain pairs $(y, \eta)$ where $y \in W$ and $\eta \in P_8$. The repleteness condition on $\la$ yields restrictions on the pairs $(y, \eta)$. Namely $\la$ must integrable with respect to $\al^\vee$ for the $\al \in \D_{0, +}$ and non integral with respect to all the other positive coroots.

Since $\D_0$ in this case consists of three copies of the root system of $\sll_2$, the condition on $\eta$ is that $(\eta, \al) = 0$ for exactly three nonconsecutive simple roots $\al$ of $\widehat{\sll}_{11}$. Since $\eta \in P_8$, we must have $(\eta, \al) = 1$ for each of the other eight simple roots $\al$. Up to cyclic permutation, the possible weights $\eta$ are exactly the following:
\begin{align*}
&(1, 1, 1, 1, 1, 1, 0, 1, 0, 1 \mid 0), \quad
(1, 1, 1, 1, 1, 0, 1, 1, 0, 1 \mid 0), \\
&(1, 1, 1, 1, 0, 1, 1, 1, 0, 1 \mid 0), \quad
(1, 1, 1, 0, 1, 1, 1, 1, 0, 1 \mid 0), \\
&(1, 1, 0, 1, 1, 1, 1, 1, 0, 1 \mid 0), \quad
(1, 1, 1, 1, 0, 1, 1, 0, 1, 1 \mid 0), \\
&(1, 1, 1, 0, 1, 1, 1, 0, 1, 1 \mid 0).
\end{align*}
There are again $7$ such weights, and we have listed them in the order in which they correspond to the weights for $W(\sll_3[3], 3/8)$. Indeed if we examine the first weight in the list above, the lengths of the blocks of ``$1$'' are $(6, 1, 1)$, which corresponds to the weight $(6, 1 \mid 1)$.

Finally we consider $W(A_{19}[8, 8, 3], 19/8)$. Using the left adjusted pyramid gives $\D_0$ isomorphic to $\D(\sll_3)^3 \oplus \D(\sll_2)^5$. By similar arguments as above, an admissible weight $\la = y(\rho - (19/8)\eta) - \rho$ is replete if and only if $(\eta, \al) = 0$ for simple roots $\al$ which come in (1) three pairs of consecutive simple roots, and (2) five other simple roots, such that no block of simple roots of length greater than $2$ arises. The weight $\eta$ must look, for example, like
\[
(\cdots, 0, 0, \cdots, 0, 0, \cdots, 0, 0, \cdots, 0, \cdots, 0, \cdots, 0, \cdots, 0, \cdots \mid 0),
\]
where the $\cdots$ are all strings of positive integers. Notice that there are exactly $8$ strings $(\cdots)$ here, and so since $\eta \in P_8$, each of these strings must consist of a single ``$1$''. One such weight is
\[
(1, 0, 1, 0, 1, 0, 1, 0, 1, 0, 1, 0, 0, 1, 0, 0, 1, 0 \mid 0),
\]
which corresponds under our bijections to $(1, 1, 1, 1, 1, 1, 0, 1, 0, 1 \mid 0)$ and to $(6, 1, \mid 1)$.


\subsection{Equality of $q$-characters}

From Proposition \ref{prop:H.la.char.thm} we now deduce the following.
\begin{thm}\label{thm:q-char-equalities}
The $q$-characters of irreducible modules related under the bijections of Theorem \ref{thm:bijections} coincide.
\end{thm}

\begin{proof}
We reduce what has to be proved to a combinatorial question about pyramids. Specifically, for $\g = \sll_n$ where $n = mu+s$, let us denote $\dim(\g_k)$ by $\dim(m, k)$. Then we are to prove that
\[
\dim(m, k) + \dim(m, u-k) - \dim(\g_0)
\]
is constant in $m$. We shall prove this by induction on $m$. Notice that $\delta(m, k) = \dim(m+1, k)-\dim(m, k)$ counts pairs of $k$-separated boxes in the pyramid in which one or both of the boxes lies in the bottom row. By separating into four cases, it is easy to see that
\begin{align*}
\delta(m, k) = {} & (u-k) + 2(m+1)\max\{0, s-k\} \\
&+ 2m\max\{0, (u-s)-k\} + (2m+1)\min\{s, u-s, k, u-k\}.
\end{align*}
Some elementary manipulations yield
\[
\delta(m, k) + \delta(m, u-k) = 2(mu+s) + u,
\]
which is also the difference between the values of $\dim(\g_0)$ for $n = (m+1)u+s$ and $n=mu+s$. So we are done.
\end{proof}

The equality of $q$-characters has technical consequences that will be useful for us below.
\begin{cor}\label{cor:self.dual}
The exceptional $W$-algebras $W(\sll_n[u^m, s], n/s)$ are self dual.
\end{cor}

\begin{proof}
By a theorem of Li \cite{Li.94} the vector space of invariant bilinear forms on a conformal vertex algebra $V$ is isomorphic to the vector space $V_0 / L_1 V_1$. For $m=0$ our $W$-algebra is of principal type, hence strongly generated by vectors of conformal weight $2$ and higher, so $V_0 = \C\vac$ and $V_1=0$. By equality of $q$-characters we have $V_0 = \C\vac$ and $V_1=0$ in all the other cases too. So in all cases there is a $1$-dimensional space of invariant bilinear forms. These vertex algebras are simple, so any nonzero invariant bilinear form is non degenerate, and so we are done.
\end{proof}

\begin{rem}\label{rem:equal.q.char}
If two necklaces are related by order reversal (for example $(4, 3 \mid 1)$ and $(3, 4 \mid 1)$ above) then the corresponding irreducible modules over the $W$-algebra have the same $q$-character. In fact this operation corresponds to duality at the level of modules.
\end{rem}

\begin{rem}
The $q$-characters of irreducible $W$-algebra modules can be interpreted in terms of counting of cylindrical partitions \cite{FW2016} and planar partitions \cite{PW2024}. Perhaps Theorem \ref{thm:q-char-equalities} has a combinatorial interpretation in these terms.
\end{rem}

\subsection{Equivalence of modular data}\label{sec:type.A.moddata}

We now set about proving the following theorem.
\begin{thm}\label{thm:S-mat.equal}
The bijections of Theorem \ref{thm:bijections} induce equivalences of modular data. In particular the vertex algebras $W(\sll_n[u^m, s], n/s)$ where $n = mu+s$, have the same fusion rules for all $m$.
\end{thm}
The equalities of $S$-matrices and of $T$-matrices will be proved in Sections \ref{subsec:S-matrices} and \ref{subsec:T-matrices}, respectively.
\begin{rem}
Let $V$ be a rational vertex algebra, and suppose the conformal dimensions of the irreducible $V$-modules are distinct from one another. Then the normalised $q$-characters of these modules are linearly independent, and so the $S$-matrix of $V$ can be deduced from the $S$-transformation of these normalised characters. If this condition holds for two of our exceptional $W$-algebras, then the equality of $S$-matrices can be deduced from the equality of $T$-matrices together with the equality of (unnormalised) characters. However, in general, distinct irreducible modules might share the same $q$-character, as happened for $W(\sll_3[3], 3/8)$ in Remark \ref{rem:equal.q.char}.
\end{rem}

The fusion rules of principal $W$-algebras have been determined in \cite{FKW}, so Theorem \ref{thm:S-mat.equal} reduces the determination of fusion rules of all boundary $W$-algebras in type $A$ to this known case. The consequences of Theorem \ref{thm:S-mat.equal} go beyond the boundary $W$-algebras, however, since the fusion rules of exceptional $W$-algebras in general tend to exhibit the ``factorisation'' phenomenon alluded to in the introduction. The fusion rules in general are expressed as products of fusion rules of a simple affine vertex algebra with those of a boundary $W$-algebra. We make this more precise in the following corollary.

To state the corollary, we recall the notation $\CF(V)$ for the Grothendieck group of the modular tensor category of $V$-modules, introduced in Section \ref{subsec:mod.data}.
\begin{cor}
Let $n = um+s$ where $1 \leq s \leq u-1$ and $s$ is coprime to $u$. Assume further that $u$ and $s$ are odd. Then
\[
\CF(W(\sll_n[u^m, s], p/u)) \cong \CF(L_{u-s}(\sll_s))^{\text{int}} \otimes \CF(L_{p-n}(\sll_n)).
\]
\end{cor}

For $k \in \Z_+$ the irreducible $L_k(\g)$-modules are $L_k(\mu)$ where $\mu$ runs over $P_+^k$. The notation $\CF(L_{k}(\g))^{\text{int}}$ employed in the theorem statement means the subring spanned by those modules for which $\mu$ lies in $Q \cap P_+^k$.

\begin{proof}
The assumption on the parity of $u$ and $s$ guarantees that the Dynkin grading of $\sll_n$ corresponding to $\mbo_u$ is a good even grading. It follows that $W = W(\sll_n[u^m, s], p/u)$ is a self-dual vertex algebra {\cite[Proposition 4.2]{AEM2024}}. Since it is also lisse and rational, Huang's theorem \cite{Huang-rigidity} assures us that the fusion rules are computed by the Verlinde formula applied to the $S$-matrix \eqref{eq:S.orig}.

The set of irreducible $W$-modules consists of $H_{f, -}^0(L_k(\la))$ as $\la$ ranges over a set $\laset$ of representatives of $\prin_\circ^k / \sim$ from $P_{0, +}$. By {\cite[Theorem 8.7]{AE2023}} this set is in bijection with $P^{p, \text{reg}} \times \check{P}^{u}_{+, f} / \wtil{W}_+$, where the action of $\wtil{W}_+$ on $(\nu, \eta)$ is by simultaneous rotation of the labeled affine Dynkin diagrams corresponding to the two weights.

Since $u$ is coprime to $n = mu + s$ we make use of the $\wtil{W}_+$-symmetry to choose representatives $\la \in \laset$ such that $\eta \in Q$. This choice causes the formula \eqref{eq:aB} for $a(\la, \la')$ to simplify. Specifically the ``cross terms'' $e^{-2\pi i [(\nu, \beta') + (\nu', \beta)]}$ take the value $1$, and the $S$-matrix coefficient $S_{\la, \la'}$ given by formula \eqref{eq:S.orig} factors into (a constant times) a product of terms
\begin{align*}
S^{(1)}_{\eta, \eta'} &= \varepsilon(y) \varepsilon(y') \sum_{w \in W(\Gamma)} \varepsilon(w) e^{-2\pi i \frac{p}{u} (\beta, w(\beta'))} \prod_{\al \in \D_{0, +}} \frac{(w(\al), \xi)}{(\al, \xi)} \\
\text{and} \quad S^{(2)}_{\nu, \nu'} &= \sum_{w \in W} \varepsilon(w) e^{-2\pi i \frac{u}{p} (w(\nu), \nu')}
\end{align*}
(abusing notation here slightly, as $S^{(1)}_{\eta, \eta'}$ depends not only on $\eta$ and $\eta'$ but also on $y$ and $y'$ through a sign). The second of these sums is essentially the $(\nu, \nu')$ entry of the $S$-matrix of the simple affine vertex algebra $L_{p-n}(\sll_n)$ \cite{KP84}. More precisely, if $K_{\nu, \nu'}$ denotes the entry of the $S$-matrix for the pair of irreducible $L_{p-n}(\sll_n)$-modules $L_{p-n}(\nu-\rho)$ and $L_{p-n}(\nu'-\rho)$ then
\[
K_{\nu, \nu'} = \frac{i^{|\D_+|}}{|P/pQ|^{1/2}} \sum_{w \in W} e^{-\frac{2\pi i}{p} (w(\nu), \nu')}.
\]
These coefficients lie in the cyclotomic field $\Q(\zeta_N)$ where $N = pn$, let $\sigma_u \in \Gal(\Q(\zeta_N)/\Q)$ denote the automorphism determined by $\sigma_u(\zeta_N) = \zeta_N^u$. Then we see that $S^{(2)}_{\nu, \nu'} = \sigma_u(K_{\nu, \nu'})$. Since the output of the Verlinde formula are integers, the Galois twist does not affect the fusion rules.

Similarly $S^{(1)}_{\eta, \eta'}$ is given by a Galois twist of the $S$-matrix coefficient of the boundary $W$-algebra $W(\sll_n[u^m, s], n/u)$ as in Proposition \ref{prop:S-matrix-boundary}. Due to Theorem \ref{thm:S-mat.equal} this $S$-matrix coincides with that of the principal $W$-algebra $W(\sll_s[s], s/u)$. The fusion rules of this principal $W$-algebra are well known {\cite[Theorem 4.3']{FKW}} (see also \cite[Section 8]{AE2019}) to be given by $\CF(L_{u-s}(\sll_s))^{\text{int}}$.

\end{proof}

\subsubsection{Equality of $S$-matrices}\label{subsec:S-matrices}

In Propositions \ref{prop:S-matrix} and \ref{prop:S-matrix-boundary} we have given the entries of the $S$-matrix in terms of a sum over the set $W(\Gamma)$. We shall now manipulate this sum into a more manageable form via a number of auxiliary lemmas. We consider $S(\h^*) = \C[\h]$, in which we may regard roots $\al \in \D$ as (linear) polynomials.
\begin{lemma}\label{lem:Macdonald}
As elements of $\C[\h]$ we have
\begin{align}\label{eq:Macdonald}
\frac{1}{|W|} \sum_{w \in W} \varepsilon(w) \prod_{\al \in \D_+} w\al = \prod_{\al \in \D_+} \al.    
\end{align}
\end{lemma}

\begin{proof}
For each $w \in W$, the set $\{w(\al) \mid \al \in \D_+\}$ is obtained by taking $\D_+$ and replacing some roots $\al$ with $-\al$. The number of negative roots equals the length $\ell(w)$ of $w$. But $\varepsilon(w) = (-1)^{\ell(w)}$, so $\prod_{\al \in \D} w\al = \varepsilon(w) \prod_{\al \in \D_+} \al$. The equality asserted by the lemma, follows immediately.
\end{proof}

Let us consider the expression $(\beta, w(\beta'))$ which appears in \eqref{eq:S.orig.firstsimp}. We have chosen things so that $\beta = -y(\eta)$ pairs trivially with all $\al \in \D_{0}$ (and similarly for $\beta'$). Therefore $(\beta, w(\beta')) = (\beta, \beta')$ for all $w \in W_0$. Combining this observation with Lemma \ref{lem:Macdonald} above, allows us to change the range of summation in \eqref{eq:S.orig.firstsimp} from $W(\Gamma)$ to $W$. More precisely, we have:
\begin{prop}\label{prop:S.allW}
For $\la$, $\la' \in \laset$, with attendant choices as above, we have
\begin{align}\label{eq:S.allW}
S_{\la, \la'} = C(\la, \la') \sum_{w \in W} \varepsilon(w) e^{-2\pi i \frac{n}{u} (\beta, w(\beta'))} \prod_{\al \in \D_{0, +}} \frac{(w(\al), \xi)}{(\al, \xi)}.
\end{align}
\end{prop}

\begin{proof}
Let $w \in W(\Gamma)$, and consider
\begin{align*}
\prod_{\al \in \D_{0, +}} (w(\al), \xi) 
&= \prod_{\al \in \D_{0, +}} (\al, w^{-1}(\xi)) \\
&= \frac{1}{|W_0|} \sum_{w_0 \in W_0} \varepsilon(w_0) \prod_{\al \in \D_{0, +}} (w_0\al, w^{-1}(\xi)),
\end{align*}
where in the second line we used Lemma \ref{lem:Macdonald}. Invariance of $\beta'$ under $W_0$ means we can replace $(\beta, w(\beta'))$ with $(\beta, ww_0(\beta'))$. Clearly \eqref{eq:S.orig.firstsimp} reduces to \eqref{eq:S.allW}.
\end{proof}

\begin{rem}
The set $W(\Gamma)$ consists of the shortest element in each coset $wW_0$ of $W_0 \subset W$. In particular $|W(\Gamma)| = |W| / |W_0|$. 
\end{rem}

\begin{lemma}\label{lem:W0Wf}
The set $\{w \in W \mid w(\D_0) \cap \D^f = \emptyset\}$ coincides with $W^f W_0$.
\end{lemma}

\begin{proof}
Let $\al = \al_{ij} \in \D_0$, which is to say $i$ and $j$ are in the same column. If $w = w^f w_0$ where $w^f \in W^f$ and $w_0 \in W_0$, and we think of these as permutations of the labels $\{1, 2, \ldots, n\}$ of the boxes, then $w_0$ will keep $i$ and $j$ in the same column, and $w^f$ cannot move them into the same row, i.e., $w(\al) \notin \D^f$. Conversely, if $w$ is some permutation of the labels which does \emph{not} send two labels in one column to a single row, it is not difficult to see that $w$ can be factored into a product consisting of a permutation preserving columns (an element of $W_0$) followed by a permutation preserving rows (an element of $W^f$).
\end{proof}

\begin{cor}\label{cor:hf.simplification}
Let $\xi \in \h^f$. For $w \in W$ outside the subset $W^f W_0$, there exists $\al \in \D_{0, +}$ such that $(w(\al), \xi) = 0$. 
\end{cor}

\begin{proof}
Recall that $\D^f = \{\al \in \D \mid \al(\h^f) = 0\}$. Therefore the condition $\prod_{\al \in \D_{0, +}} (w(\al), \xi) = 0$ on $w \in W$ is the same as the condition that the set $w(\D_0) \cap \D^f$ be non empty.
\end{proof}

Due to Corollary \ref{cor:hf.simplification}, we now obtain a significant reduction in the formula \eqref{eq:S.allW} for the $S$-matrix.
\begin{prop}
The entries of the $S$-matrix of $W(\sll_n[u^m, s], n/u)$ are given by 
\begin{align}\label{eq:Swf}
S_{\la, \la'} = C(\la, \la') \sum_{w \in W^f} \varepsilon(w) e^{-2\pi i \frac{n}{u} (\beta, w(\beta'))}.
\end{align}
\end{prop}

\begin{proof}
If $m = 0$ then we are in the principal case, $W^f = W$, and the formula has been known since \cite{FKW}. Suppose now that $m \geq 1$, then $\h^f \neq 0$, so we may specialise $\xi \in \h^f$. Because of Corollary \ref{cor:hf.simplification}, the expression \eqref{eq:S.allW} becomes
\begin{align*}
S_{\la, \la'} &= C(\la, \la') \sum_{w^f w_0 \in W^f W_0} \varepsilon(w^f w_0) e^{-2\pi i \frac{n}{u} ([w^f w_0](\beta), \beta')} \prod_{\al \in \D_{0, +}} \frac{(w^f w_0(\al), \xi)}{(\al, \xi)} \\
&= C(\la, \la') \sum_{w^f w_0 \in W^f W_0} \varepsilon(w^f w_0) e^{-2\pi i \frac{n}{u} (w^f(\beta), \beta')} \prod_{\al \in \D_{0, +}} \frac{(w_0(\al), \xi)}{(\al, \xi)},
\end{align*}
where we have used the condition $\xi \in \h^f$ to see that $w^f(\xi) = \xi$, and we have used $(\beta, \D_{0, +}) = 0$ to see that $w_0(\beta)=\beta$. Next we use \eqref{eq:Macdonald} to reduce this to \eqref{eq:Swf}.
\end{proof}

Let $R$ denote the set of rows of the pyramid, and $n_r$ the length of the row $r \in R$. Thus $n_r = s$ for one element of $R$ and $n_r = u$ for the remaining elements. Denote
\[
\D^{(r)} = \{\al_{ij} \mid \text{boxes $i$ and $j$ are in row $r$}\},
\]
and $W^{(r)} = \left<r_\al \mid \al \in \D^{(r)}\right>$. Clearly $(\al, \al') = 0$ whenever $\al \in \D^{(r)}$ and $\al'\in \D^{(r')}$ for $r'\neq r$, and
\[
W^f \cong \prod_{r \in R} W^{(r)} \quad \text{and} \quad W^{(r)} \cong S_{n_r},
\]
with the actions of the various components on $\h^*$ commuting.

Consider orthogonal projections $\pi_r : \h^* \rightarrow (\h^*)^{(r)}$ where $(\h^*)^{(r)}$ denotes the span in $\h^*$ of $\D^{(r)}$. Because of mutual orthogonality of the sets $\D^{(r)}$ we see that $W^{(r)}$ fixes $\pi_{r'}\beta$ whenever $r'\neq r$. Denote the projection to $\h^f$ by $\pi_f$. Then
\[
\pi_f + \sum_r \pi_r = \id,
\]
is an orthogonal decomposition. This is because by definition $\al \in \D^f$ means $\al$ is orthogonal to $\h^f$, and it's easy enough to see that $\dim(\h^f)=m$ and so by counting dimensions the spans of $\D^f$ and $\h^f$ make up all of $\h$.

At the risk of abusing notation, we write $\pi_s$ for the projection to the remainder part, i.e., to the component $\D^{(r)} \cong \D(\sll_s)$, and $\pi_u$ for the sum of the projections to the other components, i.e., to the components $\D^{(r)} \cong \D(\sll_u)$. So in particular $\pi_s + \pi_u + \pi_f = \id$.

As a consequence of the above remarks, we obtain a factorisation of the Weyl group sum
\begin{align}\label{eq:Swf.factored}
\sum_{w \in W^f} \varepsilon(w) e^{-2\pi i \frac{n}{u} (\beta, w(\beta'))} = e^{-2\pi i \frac{n}{u} (\pi_f\beta, \pi_f\beta')} \prod_{r \in R} \sum_{w \in W^{(r)}} \varepsilon(w) e^{-2\pi i \frac{n}{u} (\pi_r\beta, w(\pi_r\beta'))}.
\end{align}

The following lemma will allow us to effectively control the contribution of $\D^{(u)}$ to the product \eqref{eq:Swf.factored}, allowing us to relate the $S$-matrix entries $S_{\la, \la'}$ to the contribution of $\D^{(s)}$, which will ultimately recover the $S$-matrix of $W(\sll_s[s], s/u)$.
\begin{lemma}\label{lem:u.proj.rho}
Let $r \in R$ and denote by $\rho^{(r)} \in (\h^*)^{(r)}$ the Weyl vector of $\D^{(r)}_+ = \D^{(r)} \cap \D_+$. Let $\la \in \laset$ and $\beta$ associated with $\la$ as above. If $n_r = u$ then $\pi_r(\beta)$ lies in the $W^{(r)}$-orbit of $\rho^{(r)}$.
\end{lemma}

\begin{proof}
By definition of orthogonal projection, $(\pi_r\beta, \al) = (\beta, \al)$ for all $\al \in \D^{(r)}$. Certainly $(\beta, \al) \in \Z$ for all $\al \in \D$, so $\pi_r\beta \in P^{(r)}$ (the weight lattice of $\D^{(r)}$).

Since $\D^f$ is disjoint from $\D_0$ and $\beta$ is non-integrable with respect to roots in $\D_{>0}$, we have $(\beta, \al) \neq 0$ for all $\al \in \D^{(r)}$ and so $\pi_r\beta$ is a regular integral weight. Consider the unique dominant element $\mu$ of $W^{(r)}(\pi_r\beta)$. We claim $\mu = \rho^{(r)}$, which we will prove by assuming the contrary and deducing a contradiction. Indeed we would then have $(\theta^{(r)}, \mu) \geq u$, where $\theta^{(r)}$ is the highest root of $\D^{(r)}_+ \subset \D_+$. It would follow that $(\theta, \mu) \geq u$, which would be a contradiction to our choice of $\eta \in P_+^{u}$.
\end{proof}
It follows from the lemma that
\begin{align*}
\sum_{w \in W^{(u)}} \varepsilon(w) e^{-2\pi i \frac{n}{u} (\pi_u\beta, w(\pi_u\beta'))} = C(\la, \la') \sum_{w \in W^{(u)}} \varepsilon(w) e^{-2\pi i \frac{n}{u} (\rho^{(u)}, w(\rho^{(u)}))},
\end{align*}
where $C(\la, \la') = \varepsilon(\la) \varepsilon(\la')$ and $\varepsilon(\la) = \varepsilon(w)$ for that $w \in W^{(u)}$ such that $\pi_u\beta = w(\rho^{(u)})$, etc.


We consider now the contribution of the $W^{(s)}$-sum to $S_{\la, \la'}$. Indeed
\begin{align}
\sum_{w \in W^{(s)}} \varepsilon(w) e^{-2\pi i \frac{n}{u} (\pi_s\beta, w(\pi_s\beta'))} 
&= \sum_{w \in W^{(s)}} \varepsilon(w) e^{-2\pi i m (\pi_s\beta, w(\pi_s\beta'))} e^{-2\pi i \frac{s}{u} (\pi_s\beta, w(\pi_s\beta'))} \nonumber \\
&= e^{-2\pi i m (\pi_s\beta, \pi_s\beta')} \sum_{w \in W^{(s)}} \varepsilon(w) e^{-2\pi i \frac{s}{u} (\pi_s\beta, w(\pi_s\beta'))}. \label{eq:W.s.reduction}
\end{align}
In the second step we have used that $w(\beta')-\beta' \in \D^{(s)}$ and $(\pi_s\beta, \al) \in \Z$ for $\al \in \D^{(s)}$.


Let us denote
\begin{align}\label{eq:sfform.def}
\langle \beta, \beta' \rangle = m (\pi_s\beta, \pi_s\beta') + \frac{n}{u} (\pi_f\beta, \pi_f\beta')
\end{align}
Then what we have shown so far is that
\[
S_{\la, \la'} = C(\la, \la') e^{-2\pi i \langle \beta, \beta' \rangle} \sum_{w \in W^{(s)}} \varepsilon(w) e^{-2\pi i \frac{s}{u} (\pi_s\beta, w(\pi_s\beta'))}.
\]
We remark that the map $\beta \mapsto \pi_s \beta$ corresponds to the bijection of Theorem \ref{thm:bijections}, and the sums on the right hand side of \eqref{eq:W.s.reduction} are nothing other than the components of the $S$-matrix of $W(\sll_s[s], s/u)$. So (in order to apply Proposition \ref{lem:S.untwistable} and deduce Theorem \ref{thm:S-mat.equal}) it remains to prove that the factor $e^{-2\pi i \langle \beta, \beta' \rangle}$ takes the form $C_0 \varepsilon(\beta) \varepsilon(\beta')$ for some constant $C_0$ and function $\varepsilon : \beset \rightarrow \{+1, -1\}$. This we shall do in the following proposition.
\begin{prop}\label{prop:s.f.relation} Let $\beta, \beta' \in \beset$. Then
\begin{enumerate}
\item $\langle \beta, \beta' \rangle = \frac{1}{2}\kappa_{\g_0}(\pi_u\beta, \pi_u\beta')$,

\item There exists a constant $C_0$ and a function $\varepsilon : \beset \rightarrow \{+1, -1\}$ such that
\[
e^{-2\pi i \langle \beta, \beta' \rangle} = C_0 \varepsilon(\beta) \varepsilon(\beta').,
\]
\end{enumerate}
\end{prop}

\begin{proof}
First we shall deduce part (2) from part (1). First we observe that $\frac{1}{2}\kappa_{\g_0}(\gamma, \gamma') \in \Z$ whenever $\gamma, \gamma' \in Q$ since the form is a sum of evaluations of roots (namely $\gamma$ and $\gamma'$) against coroots (those of $\g_0$), and coroots $\al^\vee$ and $-\al^\vee$ contribute equally to the sum.

By Proposition \ref{lem:u.proj.rho}, for each $\beta$ there exists $\gamma \in Q$ such that $\pi_u\beta = \rho^{(u)} - \gamma$, and similarly $\pi_u\beta' = \rho^{(u)} - \gamma'$. Furthermore $\rho^{(u)} \in \frac{1}{2}\Z\D^{(u)} \subset \frac{1}{2}Q$. By the first half, then,
\begin{align*}
\langle \beta, \beta' \rangle = \frac{1}{2}\left( \kappa_{\g_0}(\rho^{(u)}, \rho^{(u)}) - \kappa_{\g_0}(\rho^{(u)}, \gamma + \gamma') + \kappa_{\g_0}(\gamma, \gamma') \right).
\end{align*}
But now $\frac{1}{2}\kappa_{\g_0}(\gamma, \gamma')$ is an integer, so
\[
e^{-2\pi i \langle \beta, \beta' \rangle} = C_0 \varepsilon(\beta) \varepsilon(\beta')
\]
where $C_0 = e^{-\pi i \kappa_{\g_0}(\rho^{(u)}, \rho^{(u)})}$ and $\varepsilon(\beta) = e^{-\pi i \kappa_{\g_0}(\rho^{(u)}, \gamma)} \in \{+1, -1\}$.

Now we prove part (1). By Lemma {\ref{lem:killing.relation}} and Lemma {\ref{lem:swapper} part (1),} below, we have
\begin{align}\label{eq:s.f.relation.1}
\langle \beta, \beta' \rangle = \frac{1}{2}\kappa_{\g_0}(\pi_s\beta, \pi_s\beta') + \frac{1}{2}\kappa_{\g_0}(\pi_f\beta, \pi_f\beta').
\end{align}
Since $\beta$ annihilates the roots of $\g_0$ we have $\beta|_{\g_0} = 0$ and so $\kappa_{\g_0}(\beta, \beta') = 0$. It follows that we can replace $\pi_f \beta$ by $-(\pi_s+\pi_u) \beta$ in \eqref{eq:s.f.relation.1}, obtaining
\begin{align}\label{eq:s.f.relation.3}
\langle \beta, \beta' \rangle = \frac{1}{2}\kappa_{\g_0}(\pi_s\beta, \pi_s\beta') + \frac{1}{2}\kappa_{\g_0}((\pi_u+\pi_s)\beta, (\pi_u+\pi_s)\beta').
\end{align}
Now by Lemma {\ref{lem:swapper} part (2)}, this in turn is reduced to
\begin{align}\label{eq:s.f.relation.4}
\langle \beta, \beta' \rangle = \frac{1}{2}\kappa_{\g_0}(\pi_u\beta, \pi_u\beta').
\end{align}
\end{proof}


Finally we prove the lemmas used in the proof of Proposition \ref{prop:s.f.relation}.
\begin{lemma}\label{lem:killing.relation}
For all $x, y \in \h^f$ the following equality holds
\[
\frac{n}{u} (x, y) = \frac{1}{2} \kappa_{\g_0}(x, y).
\]
\end{lemma}

\begin{proof}
The universal $W$-algebra $W^k(\g, f)$ has a set of strong generators indexed by a graded basis of $\g^f$, with conformal weight $1+j$ assigned to those generators in $\g_{-j}^f$. In particular $W^k(\g, f)$ contains a Heisenberg subalgebra $H^\phi(\h^f)$. The bilinear form $\phi : \h^f \times \h^f \rightarrow \C$ was determined explicitly in \cite[Theorem 2.4(c)]{KRW03} as
\begin{align}\label{eq:gnat.form}
\phi(x, y) = k (x, y) + \frac{1}{2} \left( \kappa_{\g}(x, y) - \kappa_{\g_0}(x, y) \right),
\end{align}
where $(-, -)$ denotes the restriction of the standard normalised bilinear form on $\g$, and the $\kappa$ are Killing forms.

From the character formula \eqref{eq:H.la.char.thm} (see the proof of Corollary \ref{cor:self.dual}) we deduce that $W_k(\g, f)_1 = 0$. Therefore $W^k(\g, f)_1$ lies inside the maximal ideal, which can only occur if $\phi = 0$. Accordingly, setting the right hand side of \eqref{eq:gnat.form} to zero, and using $\kappa_\g(-,-) = 2h^\vee (-,-)$, we obtain the desired equation.
\end{proof}

\begin{lemma}\label{lem:swapper}
Identifying $\h^*$ with $\h$ as usual, we have:
\begin{enumerate}
\item For all $\la, \la' \in (\h^*)^{(s)}$ we have
\[
m (\la, \la') = \frac{1}{2}\kappa_{\g_0}(\la, \la').
\]

\item For all $\beta, \beta'$ orthogonal to $\D_0$ we have
\[
2\kappa_{\g_0}(\pi_s\beta, \pi_s\beta') + \kappa_{\g_0}(\pi_u\beta, \pi_s\beta') + \kappa_{\g_0}(\pi_s\beta, \pi_u\beta') = 0.
\]
\end{enumerate}
\end{lemma}

\begin{proof}
First we prove part (1). Under the conventions on nilpotent elements in $\sll_n$ and pyramids which we have adopted from \cite{CMcG}, the vector space $\h^{(s)}$ is identified with the space of traceless diagonal $n \times n$ matrices with nonzero entries permitted only in the following $s$ positions: $E_{1, 1}$, $E_{(m+1)+1, (m+1)+1}$, $E_{2(m+1)+1, 2(m+1)+1}$, and so on up to $E_{(s-1)(m+1)+1, (s-1)(m+1)+1}$.

Identifying an element $h = \sum h_i E_{i(m+1)+1, i(m+1)+1} \in \h^{(s)}$ with a vector $(h_0, \ldots, h_{s-1})$, the normalised bilinear form accordingly becomes $(h, h') = \sum_i h_i h_i'$.

The subalgebra $\g_0$ consists of certain block diagonal matrices, specifically $s$ blocks of size $(m+1) \times (m+1)$, followed by $u-s$ blocks of size $m \times m$. Note in particular that the diagonal matrices listed above are precisely the $(1, 1)$-entries of the first $s$ blocks. It follows that
\[
\tr_{\g_0} \ad(h) \ad(h') = 2 m \sum_{i} h_i h'_i.
\]
from which we immediately derive the stated relation.

Now we turn to part (2).  Since $\beta'$, by hypothesis, annihilates the root vectors in $\g_0$, we have
\[
\kappa_{\g_0}(\pi_s\beta, \beta') = 0.
\]
It follows that
\[
\kappa_{\g_0}(\pi_s\beta, (\pi_s+\pi_u)\beta') = \kappa_{\g_0}(\pi_s\beta, -\pi_f\beta').
\]
Again following our conventions on pyramids, each vector space $\h^{(r)}$ is identified with the space of traceless matrices with nonzero entries only in a certain set $\{E_{i, i} \mid i \in \Sigma_r\}$ (for example above we saw $\Sigma_s = {i(m+1)+1 \mid i=0,\ldots s-1}$). Furthermore $\h^f$ is identified with the space of traceless matrices with entries constant on each of the sets $\Sigma_r$. By a similar calculation as in the proof of part (1), we conclude that for $x \in \h^{(s)}$ and $y \in \h^f$ we have $\kappa_{\g_0}(x, y) = 0$. Therefore
\[
\kappa_{\g_0}(\pi_s\beta, \pi_s\beta') = -\kappa_{\g_0}(\pi_s\beta, \pi_u\beta').
\]
By the same argument, we also have
\[
\kappa_{\g_0}(\pi_s\beta, \pi_s\beta') = -\kappa_{\g_0}(\pi_u\beta, \pi_s\beta').
\]
Summing the two relations gives the claimed result.
\end{proof}

\subsubsection{Equality of $T$-matrices}\label{subsec:T-matrices}

In this section we will show that the conformal dimensions of the irreducible $W(\sll_n[u^m, s], n/u)$-modules (corresponding under the bijections of Theorem \ref{thm:bijections}) are independent of $m$. In particular the $T$-matrices of these vertex algebras coincide.

We recall that the conformal dimension $h_\la$ of $H_{f, -}^0(L_k(\la))$ is given by the formula {\cite[Remark 2.2]{KRW03}}
\begin{align*}
h_\la &= \frac{u}{2p}\left(|\la+\rho|^2 - |\rho|^2\right) - \frac{p}{2u}|x_0|^2 + (\rho, x_0)
\end{align*}
which, in case $\la = -(p/u)\eta$, can be rewritten as
\begin{align}\label{eq:h.la.fla}
h_\la = \frac{u}{2p}\left(|\rho-\frac{p}{u}\eta|^2 - |\rho-\frac{p}{u}x_0|^2\right). 
\end{align}

Let us write $\eta \in P_{+, f}^u$ in the form
\[
\eta = \sum_{i=0}^{u-1} \varpi_{mi + a_i},
\]
where $(a_0, a_1, \ldots, a_{u-1})$ is a fixed list of non decreasing integers. For example in $\sll_3$ the weight $(6, 1 \mid 1)$ corresponds to
\[
\mathbf{a} = (a_0, \ldots, a_7) = (1, 1, 1, 1, 1, 1, 2, 3).
\]
The corresponding weight $(1, 1, 1, 1, 1, 1, 0, 1, 0, 1 \mid 0)$ in $\sll_{11}$ corresponds to this same list $\mathbf{a}$. The bijections between $P_{+, f}^u$ preserve these lists. From the description \eqref{eq:x0.description} of the grading element $x^0$, we see that (for all $m$) it corresponds to the list $\mathbf{a} = (1, 2, \ldots, s-1, s, s, s, \ldots, s)$.

Let us fix a list $\mathbf{a}$ and, for $m \in \Z_{+}$ and $n = mu+s$, consider the weight $\eta$ for $\sll_n$ associated with $\mathbf{a}$ as above. We introduce
\[
N_m(\mathbf{a}) = \frac{u}{2n} \left|\rho - \frac{n}{u}\eta\right|^2 - \frac{n^2}{24u}.
\]
\begin{prop}\label{prop:N.const}
The number $N_m(\mathbf{a})$ depends on $\mathbf{a}$ but not on $m$.
\end{prop}

\begin{cor}\label{cor:same.confdim}
The bijections between irreducible $W(\sll_n[u^m, s], n/s)$-modules of Theorem \ref{thm:bijections}, preserve conformal dimensions.
\end{cor}

\begin{proof}
The weights $x^0$ correspond under the bijections, so the claim follows immediately from Proposition \ref{prop:N.const} and formula \eqref{eq:h.la.fla}.
\end{proof}

\begin{proof}[Proof of Proposition \ref{prop:N.const}]
First we expand
\begin{align*}
N_m(\mathbf{a}) &= \frac{u}{2n} |\rho|^2 - (\rho, \eta) + \frac{n}{2u}|\eta|^2 - \frac{n^2}{24u} \\
&= \frac{1}{24}\left[(n^2-1)u - \frac{n^2}{u}\right] - (\rho, \eta) + \frac{n}{2u}|\eta|^2,
\end{align*}
where in passing to the second line we used the Freudenthal-de Vries strange formula $|\rho|^2 = h^\vee \dim(\g) / 12$.

We recall that the standard bilinear form for $\sll_n$, relative to the basis $\{\varpi_1, \ldots, \varpi_{n-1}\}$, is given by the Gram matrix $F$ whose components are
\[
F_{ij} = \min\{i, j\} - \frac{ij}{n}.
\]
We use this to simplify $|\eta|^2$ and $(\rho, \eta)$ in turn. It will be convenient to introduce the quantities
\begin{align*}
S = \sum_{j=0}^{u-1} a_j, \quad S_2 = \sum_{j=0}^{u-1} a_j^2 \quad \text{and} \quad M = \sum_{k=1}^{u} k a_{u-k},
\end{align*}

For $|\eta|^2$ we have
\[
|\eta|^2 = \sum_{i, j=0}^{u-1} \left( \min\{mi+a_i, mj+a_j\} - \frac{(mi+a_i)(mj+a_j)}{n} \right).
\]
It is not hard to see that in the summation of $\min\{mi+a_i, mj+a_j\}$, the term $m(u-k)+a_{u-k}$ appears $2k-1$ times, and so the sum of these terms is
\[
\sum_{k=1}^u (2k-1) (m(u-k) + a_{u-k}) = m \frac{2u^3-3u^2+u}{6} + 2M - S.
\]
Combining with the second summand gives
\begin{align*}
|\eta|^2 = m \frac{2u^3-3u^2+u}{6} + 2M - S - \frac{1}{n}\left[m^2\frac{u^2(u-1)^2}{4} + mu(u-1) + S^2\right].
\end{align*}

For $(\rho, \eta)$ we have
\[
(\rho, \eta) = \sum_{i=1}^n \sum_{j=0}^{u-1} \left( \min\{i, mj+a_j\} - \frac{i(mj+a_j)}{n} \right).
\]
We compute the sum of the terms $\min\{i, mj+a_j\}$ via the following observation: for fixed $j$ and $i$ varying from $1$ to $n$, these terms assume the values $1, 2, \ldots, mj+a_j-1$ followed by $n-(mj+a_j)$ copies of $mj+a_j$. The total is thus
\[
\sum_{j=0}^{u-1} \left[ \frac{(mj+a_j)(mj+a_j-1)}{2} + (n - (mj + a_j))(mj+a_j) \right].
\]
Combining this with the other summand gives the result
\begin{align*}
(\rho, \eta) = -\frac{n-1}{2}\left(m \frac{u(u-1)}{2} + S\right) - \frac{1}{2}m^2 \frac{u(u-1)(2u-1)}{6} - m(uS-M) - \frac{1}{2}S_2.
\end{align*}

Substituting in the definition of $N_m(\mathbf{a})$ yields
\begin{align*}
N_m(\mathbf{a}) = \frac{S^2 + s(u+1)S - uS_2 - 2sM}{2 u} + \frac{s^2 \left(u^2-1\right)-u^2}{24 u},
\end{align*}
which is manifestly independent of $m$.
\end{proof}

\begin{prop}\label{prop:c.formula}
The central charge of $W(\sll_n[u^m, s], n/u)$ is given by the formula
\[
c = - \frac{(s-1)(u-s-1)(u+us-s^2)}{u}.
\]
In particular it is independent of $m$.
\end{prop}

\begin{proof}
The general formula for the central charge {\cite[Theorem 2.2]{KRW03}} at level $k = h^\vee + p/u$, specialised to the case of a nilpotent with even grading and grading element $x_0$, is
\begin{align}\label{eq:c.formula.general}
c = \dim(\g_0) - 12 \frac{u}{p} \left|\rho - \frac{p}{u}x_0\right|^2.
\end{align}
For the left adjusted pyramid we read off $\dim(\g_0) = sm(m+1) + (u-s)m(m-1) + n-1$. From this and from the explicit description of $x_0$, with some elementary algebra we deduce the formula in the theorem statement.
\end{proof}

\begin{cor}
For fixed coprime $u, s$ and varying $m \in \Z_+$, the $T$-matrices of the vertex algebras $W(\sll_n[u^m, s], n/s)$ coincide.
\end{cor}

\begin{proof}
Follows immediately from Corollary \ref{cor:same.confdim} and Proposition \ref{prop:c.formula}.
\end{proof}

\end{document}